\documentclass[reqno]{amsart}

\usepackage[utf8]{inputenc}
\usepackage{graphicx}
\usepackage{mathtools}
\usepackage{tikz}
\usepackage{tikz-cd}
\usetikzlibrary{arrows, backgrounds, calc, chains, decorations, patterns, positioning, shapes}
\usepackage{graphviz}
\usetikzlibrary{decorations.pathreplacing,decorations.markings}

\usepackage{amssymb}
\usepackage{amstext}

\usepackage{caption}
\usepackage{enumitem}
\usepackage{hyperref}
\usepackage{cleveref}
\usepackage{subcaption}
\usepackage{tikz}
\usepackage{mathrsfs}

\theoremstyle{plain}

\newtheorem{theorem}{Theorem}[section]

\newtheorem{corollary}[theorem]{Corollary}
\newtheorem{lemma}[theorem]{Lemma}

\theoremstyle{definition}
\newtheorem{definition}[theorem]{Definition}
\newtheorem{remark}{Remark}
\newtheorem{example}{Example}

\title[Cohomology ring of non-compact abelian arrangements]{Cohomology ring of non-compact abelian arrangements}
\author{Evienia Bazzocchi}
\address{
    Evienia Bazzocchi 
}
\email{evienia.bazzocchi2@unibo.it}
\author{Roberto Pagaria}
\address{
    Roberto Pagaria \newline \indent
    Dipartimento di Matematica \newline \indent
    Universit\`a di Bologna \newline \indent
    Bologna, BO, 40126, Italia
}
\email{roberto.pagaria@unibo.it}

\author{Maddalena Pismataro}
\address{
    Maddalena Pismataro
}
\email{maddalena.pismataro2@unibo.it}
\date{\today}

\DeclareMathOperator{\PD}{PD}
\DeclareMathOperator{\BM}{BM}

\DeclareMathOperator{\dime}{dim}
\DeclareMathOperator{\conf}{Conf}

\DeclareMathOperator{\rk}{rk}
\DeclareMathOperator{\sgn}{sgn}
\DeclareMathOperator{\res}{res}

\DeclareMathOperator{\gr}{gr}
\DeclareMathOperator{\Hom}{Hom}

\newcommand{\AAA}{\mathcal{A}}
\newcommand{\LL}{\mathcal{L}}
\newcommand{\RR}{\mathbb{R}}

\date{\today}

\begin{document}

\begin{abstract}
    We give a Orlik-Solomon type presentation for the cohomology ring of arrangements in a non-compact abelian Lie group.
    The new insight consists in comparing arrangements in different abelian groups. Our work is based on the Varchenko-Gel'fand ring for real hyperplane arrangements and from that we deduce the cohomology rings of all other abelian arrangements.
    As by-product, we obtain a new proof of the Orlik-Solomon relations and De Concini-Procesi ones.
\end{abstract}


\maketitle

\section{Introduction}

The topology of hyperplane arrangements is a classical subject whose study goes back to the sixties.
It is originally motivated by the relations with braid groups \cite{Deligne1972} and configuration spaces \cite{CLM76}.
The study of arrangements was also motivated by combinatorics, in particular matroid theory, and by singularity theory.

A first generalization of hyperplane arrangements -- now called toric arrangements -- were introduced in \cite{Looijenga}.
In the 2000s, De Concini and Procesi \cite{DeConciniProcesi2005} started a systematic study of toric arrangements motivated by the application to knapsack problem.
Finally, abelian arrangements were introduced in 2016 by Bibby \cite{Bibby}.

An important focus on this subject is the description of the cohomology ring of the arrangement complement. Arnold identified relations between differential forms in the cohomology ring of braid hyperplane arrangements \cite{Arnold69} and Brieskorn described the cohomology ring of arrangements of Coxeter type \cite{Brieskorn1971-1972}.
A milestone is the famous result by Orlik and Solomon \cite{OrlikSolomon1980}: they presented the cohomology ring of the complement of complex hyperplane arrangements by generators and relations.
Later, Varchenko and Gel'fand studied the cohomology ring of real hyperplane arrangements using Heaviside functions as generators \cite{VarchenkoGelfand87}.
They observed an analogy between the real and the complex case:
\begin{quote}
The relation \eqref{rel} \footnote{Equation \eqref{rel} of this article corresponds to relation (2) in \cite{VarchenkoGelfand87}} is the even analogue of a relation of Orlik and Solomon \cite{OrlikSolomon1980} for differential forms.
\end{quote}
However, they do not find a geometric connection between the two cases.

The cohomology of subspace arrangements was determined in \cite{goresky} as a module and a rational model was provided in \cite{DCP95}.
The multiplicative structure of the integral cohomology ring was later provided in \cite{FeichtnerZiegler} and in \cite{LonguevilleSchultz}.

Proudfoot \cite{Proudfoot06} relates the VG-algebra and the OS-algebra with $\mathbb{F}_2$-coefficients  overcoming the non-commutativity of the OS algebra.
Moseley \cite{Moseley17} shows an isomorphism between the VG-algebra and the cohomology of (subspace) arrangements in $(\mathbb{R}^3)^n$.
The Varchenko-Gel'fand presentation was generalized to oriented matroids and conditional oriented matroids \cite{GR89,Cordovil02,DorpalenBarry23,DBPW22}.

A Orlik-Solomon type presentation for toric arrangements was obtained by De Concini and Procesi \cite{DeConciniProcesi2005} in the unimodular case and later extended to all toric arrangements \cite{CD17,CDDMP2020}.
The cohomology of the complement can be also computed, via the Morgan model, from the cohomology of a wonderful compactification and their strata.
This approach, for toric arrangements, was developed in a sequence of articles \cite{DCG18,DCG19,DCGP20,DCG21,MociPagaria,GPS2024}.

In the more general case of abelian arrangements only few results are known: 
Bibby computed the Euler characteristic in the case of complex groups \cite{Bibby} and Liu, Tran, and Yoshinaga described the additive structure of the cohomology in the case of non-compact abelian groups \cite{LiuTranYoshinaga2021}.

Our main result is a new and uniform approach to the cohomology ring of non-compact abelian arrangements.
Our technique consists in the pushforward of the cohomological relations from the real hyperplane case (see \cite{VarchenkoGelfand87}) to the general non-compact abelian arrangement.
The main theorem (\Cref{thm: main}) is an Orlik-Solomon type presentation of the cohomology ring of non-compact abelian arrangements.

We obtain a new proof of the Orlik-Solomon result \cite{OrlikSolomon1980} for complexified hyperplane arrangements and of the De Concini-Procesi presentation \cite{DeConciniProcesi2005} for unimodular toric arrangements.
We also obtain a new proof of the presentation of the rational cohomology ring for general toric arrangements, that firstly appeared in \cite{CDDMP2020}.
In all other cases our result is original.

We start by setting the notations in \Cref{Sec:2}.
In \Cref{Sec:3}, we explicitly describe the pushforward along the map $i\colon M^\mathbb{R}(\mathcal{A}) \to M^{a,b}(\mathcal{A})$ induced by the inclusion $\mathbb{R} \hookrightarrow \mathbb{R}^b \times (S^1)^a$.
We manipulate the Varchenko-Gel'fand relations to obtain relations in the cohomology ring of central and unimodular arrangements (see Lemma \ref{lemma2}).
\Cref{Sec:4} extends the results of \cite{LiuTranYoshinaga2021} to the case of non-central arrangements (see \Cref{thm: nonC}).
Then, in \Cref{Sec:5} we use the technique of separating coverings developed in \cite{CDDMP2020} to obtain the main theorem.
In \Cref{sec:6} we consider toric arrangements and we compare the presentation in \cite{CDDMP2020} with our own; we also show how to derive the former from the latter.
Finally, in \Cref{Sec: conf} we apply our theorem to the arrangements of type $A_{n}$ improving a result of \cite{CohenTaylor1978} on configuration spaces in $\mathbb{R}^b \times (S^1)^a$.

\section{Definitions and notations}\label{Sec:2}
Consider the abelian group $G=\mathbb{R}^b\times (S^1)^a$ canonically oriented and let $e=(0,\dots,0,1,\dots, 1)$ be the unit of $G$. We denote $d:=\dime(G)-1= a+b-1$.

Let $\Lambda\cong \mathbb{Z}^r$ be a lattice with the choice of an orientation, so that $\Hom (\Lambda,G) \simeq G^r$ has a natural orientation.
A primitive character $\chi\in \Lambda$ corresponds to a group morphism $\chi \colon \Hom (\Lambda,G)\rightarrow G$.
Let $E$ be a finite set with a total ordering.
An abelian arrangement is a finite collection of connected subvarieties in $\Hom (\Lambda,G)$ defined by a finite list of primitive characters $\chi_i$ and elements $g_i \in G$ indexed by $E$: 
\[\mathcal{A}=\{H_i:=\chi_i^{-1}(g_i)\}_{i\in E}.\]
This definition, specialized to the case $G=\mathbb{C}$, corresponds to hyperplane arrangements with the property that each hyperplane has a normal vector with rational coordinates. 

The \emph{poset of layers} $\mathcal{L}(\mathcal{A})$ (or \emph{intersection poset}) of $\mathcal{A}$ is the set of all connected components of intersections of elements in $\mathcal{A}$, ordered by reverse inclusion.
An arrangement is called \emph{central} if the intersection of all the subvarieties is non-empty, and it is called \emph{unimodular} whenever all the possible intersections of elements in $\mathcal{A}$ 
is either connected or empty.

The \emph{complement} of the arrangement is:
$$M(\mathcal{A})=\Hom (\Lambda,G)\setminus \bigcup_{i\in E}H_i.$$
We also denote $M(\mathcal{A})$ by $M^{a,b}(\mathcal{A})$, when we need to specify what abelian group $G=\mathbb{R}^b\times (S^1)^a$ we are working with.
When $G=\mathbb{R}$, we write $M^{\mathbb{R}}(\mathcal{A})$ instead of $M^{0,1}(\mathcal{A})$.

Central abelian arrangements are related to matroids, oriented matroids, and arithmetic matroids, we briefly recall definitions and terminologies.
For further details we refer to the standard textbooks \cite{OxleyMatroidTheory,OrientedMatroids,AndersonOrientedMatroids} for oriented matroids, to the articles \cite{DAdderioMoci,BradenMoci} for arithmetic matroids and to \cite{OrientableArithmeticMatroids,PagariaPaolini} for the interaction between oriented and arithmetic matroids. 

Consider a central abelian arrangement $\mathcal{A}=\{H_i=\chi_i^{-1}(e)\}_{i\in E}$, the collection of $\chi_i \in \Lambda$ define a realizable, oriented and arithmetic matroid in the following way.
A \emph{matroid} $(E, \rk )$ is a pair of a groundset $E$ and a \emph{rank function} $\rk \colon 2^E \to \mathbb{N}$ satisfying the three properties:
\begin{enumerate}[label=(R\arabic*)]
    \item $\rk(A) \leq \lvert A \rvert$ for all $A \subseteq E$,
    \item (increasing) $ \rk(A) \leq \rk(B)$ for all $A \subseteq B \subseteq E$,
    \item (submodular) $ \rk(A) + \rk(B) \geq \rk(A\cap B) +\rk (A \cup B)$ for all $A, B \subseteq E$.
\end{enumerate}

For any central arrangement $\mathcal{A}$, consider the rank function $\rk_\mathcal{A}(A)= \rk_{\mathbb{Z}} \langle \chi_i \mid i \in A \rangle$ for any $A \subseteq E$; the pair $(E, \rk_\AAA) $ is a matroid.
The real codimension of the intersection $\cap_{i \in A} H_i$ is equal to $(a+b)\rk_\AAA(A)$.
A subset $A$ is \emph{independent} if $\rk(A)=\lvert A \rvert$, \emph{dependent} if $\rk(A)<|A|$, and a \emph{basis} if $\rk(A)=\lvert A \rvert = \rk(E)$.
A subset $B$ is dependent on $A \subseteq E$ if $\rk (A \cup B)=\rk(A)$.
A set $C \subseteq E$ is a \emph{circuit} if it is a minimally dependent set, i.e.\ $\rk(C)=\lvert C \rvert -1$ and for all $A \subsetneq C$ $\rk(A)=\lvert A \rvert$.
A subset $A\subseteq \mathcal{A}$ has \emph{nullity} $1$ if $\rk(A)=\lvert A \rvert -1$ and, in this case, we denote by $C(A)$ the unique circuit contained in $A$.

An \emph{arithmetic matroid} is a triple $(E,\rk,m)$ where $(E,\rk)$ is a matroid and the \emph{multiplicity function} $m \colon 2^E \to \mathbb{N}_{>0}$ satisfies:
\begin{enumerate}[label=(AM\arabic*)]
    \item $m(A \cup \{i\}) \mid m(A)$ for all $A \subseteq E$ and all $i \in E$ dependent on $A$,
    \item $m(A) \mid m(A \cup \{i\})$ for all $A \subseteq E$ and all $i \in E$ independent on $A$,
    \item \label{item:AM3}
    for all subsets $A \subseteq B \subseteq E$ such that $B$ is a disjoint union $B = A \sqcup F \sqcup T$ satisfying $\rk(C) = \rk(A) + \lvert C \cap F \rvert$ for all $A \subseteq C \subseteq B$, then
\[ m(A) m(B) = m(A \cup F ) m(A \cup T ),\]
\item for all $A \subseteq B \subseteq E$ with $\rk(A)=\rk(B)$, then
\[ \sum_{A \subseteq C \subseteq B} (-1)^{\lvert C \setminus A \rvert} m(C) \geq 0,\]
\item  for all $A \subseteq B \subseteq E$ with $\lvert A \rvert + \rk( E \setminus A)= \lvert B \rvert + \rk(E\setminus B)$, then
\[ \sum_{A \subseteq C \subseteq B} (-1)^{\lvert C \setminus A \rvert} m(E \setminus C) \geq 0.\]
\end{enumerate}

The function $m_\AAA \colon 2^E \to \mathbb{N}_{>0}$ defined by $m_\AAA(A)= \# \operatorname{tor} \Lambda/\langle \chi_i \mid i \in A \rangle$, i.e. the cardinality of the torsion of the quotient $\Lambda/\langle \chi_i \mid i \in A \rangle$ endows the matroid $(E,\rk_\AAA)$ of an arithmetic matroid structure.
Since the arrangement $\mathcal{A}$ is central, the number of connected components of the intersection $\cap_{i \in A} H_i$ is equal to $m_\AAA(A)^a$.
In particular, for $a=0$ all intersections are connected.
A subset $A \subseteq E$ is \emph{unimodular} if $m(A)=1$. 

The above definitions do not change if we change the description of some $H_i$ replacing $\chi_i$ with $-\chi_i$.
Now we fix for each subvariety $H_i$ a choice of a character $\chi_i$, the construction of the oriented matroid will depend on these choices.

A signed subset $(C^+,C^-)$ of $E$ is a pair of disjoint subset of $E$ and we denote its support by $C= C^+ \sqcup C^-$.
An \emph{oriented matroid} is a pair $(E, \mathcal{C})$ of a groundset $E$ and a collection of signed subsets of $E$ such that:
\begin{enumerate}[start=0, label=(C\arabic*)]
    \item $(\emptyset,\emptyset) \not \in \mathcal{C}$,
    \item (symmetry) if $(C^+,C^-) \in \mathcal{C}$, then $(C^-,C^+) \in \mathcal{C}$,
    \item (incomparability) for all $(C^+,C^-), (D^+,D^-) \in \mathcal{C}$ such that $C \subseteq D$ then $(C^+,C^-)= (D^+,D^-)$ or $(C^+,C^-)= (D^-,D^+)$,
    \item (circuit elimination) for all $(C^+,C^-), (D^+,D^-)$ such that $(C^+,C^-) \neq (D^-,D^+)$ and for all $i \in C^+ \cap D^-$ there exists $(B^+,B^-) \in \mathcal{C}$ such that $B^+ \subseteq (C^+ \cup D^+) \setminus \{i\}$ and $B^- \subseteq (C^- \cup D^-) \setminus \{i\}$.
\end{enumerate}
We call \emph{oriented circuits} the signed subsets of an oriented matroid. 
The set $\{ C\}_{(C^+,C^-) \in \mathcal{C}}$ is the set of circuits of a matroid and such a matroid is unique.

Consider, for each minimal linear relation $\sum_{i \in C} n_i \chi_i =0 $ among characters of $\AAA$, the sets $C^+=\{i \in E \mid n_i>0\}$ and the set $C^-=\{i \in E \mid n_i<0\}$.
The collection $\mathcal{C}_\AAA$ of all such pairs $(C^+,C^-)$ defines an oriented matroid.
Notice that for each signed circuit $(C^+,C^-)$ the linear relation is unique up to scalar multiplication, so also the pair $(C^-,C^+)$ is a signed circuit and we call it the opposite circuit of $(C^+,C^-)$.
For the sake of notation we will denote an oriented circuit by $C=C^+\sqcup C^-$.

An alternative definition of oriented matroids is given by basis orientations (also called chirotopes): let, for each basis $B \subseteq E$, $\sgn (B)$ the sign of the basis $(b_1, b_2, \dots, b_r)$ of $\Lambda \otimes_\mathbb{Z} \mathbb{Q}$, where we take the element of $B=\{b_1, \dots, b_r\}$ in the order prescribed by the ordered set $E$.
The collection of $\sgn(B)$ for all basis $B$ of the matroid determines the oriented matroid.
Changing the orientation of $\Lambda$ reverses all signs $\sgn(B)$, but does not affect the oriented matroid.

The situation for non-central arrangements is different.
Let $\mathcal{A}=\{H_i=\chi_i^{-1}(g_i)\}$ be an abelian arrangement, the collection of characters $\chi_i$ for $i \in E$ defines an oriented arithmetic matroid as in the central case.
For $A \subseteq E$ if the intersection $\cap_{i \in I} H_i$ is non-empty, then it has codimension $(a+b)\lvert A \rvert$ and $m(A)^a$ connected components.
If $A$ is an independent set, the intersection is certainly non-empty. Otherwise, it depends on the elements $g_i$ for $i \in A$.
We say that a dependent set is \emph{central} if the corresponding intersection is non-empty.

\medskip
In order to define certain cohomological classes in $H^{a+b-1}(G\setminus \{e\})$ we consider the immersion of a small sphere $i\colon S^{a+b-1}\hookrightarrow G\setminus e$ centered in $e$. Let $W$ be the top dimensional class generating $H_{a+b-1}(S^{a+b-1})$. We denote by $\omega$ the class in $H^{a+b-1}(G\setminus \{e\})$ dual to $i_*(W)$ with the orientation given by the outer normal first rule.

For $j=1,\cdots, a$, let $Y^j$ be the standard generators of $H_1(S^1)$ and let $Y$ be the top class in $H_a((S^1)^a)$. For $x\in \mathbb{R}^b\setminus\{0\}$ we consider the immersion
\begin{align*}
i_x \colon \quad (S^1)^a  &\hookrightarrow G\setminus\{e\}
\\
z &\mapsto (x,z),
\end{align*}
Up to homotopy, $i_x$ only depends on the connected component of $\mathbb{R}^b \setminus \{0\}$ the point $x$ belongs to. 
In the case $b=1$ we choose $x<0$, otherwise the choice of the point does not matter.
Denote by $\psi^j$ the class in $H^1(G\setminus\{e\})$ dual to $(i_x)_*(Y^j)$  and by $\psi$ the class in $H^a(G\setminus\{e\})$ dual to $(i_x)_*(Y)$. 

For any $i\in E$, the map $\chi_i$ induces $(\chi_i-g_i)|_{M^{a,b}(\mathcal{A})}\colon M^{a,b}(\mathcal{A})\rightarrow G\setminus\{e\}$. So we set
\begin{gather*}
    \omega_i=(\chi_i-g_i|_{M^{a,b}(\mathcal{A})})^*(\omega)\in H^{a+b-1}(M^{a,b}(\mathcal{A})), \\
    \psi_i^j=(\chi_i|_{M^{a,b}_{\mathcal{A}}})^*(\psi^j)\in H^1(M^{a,b}(\mathcal{A})), \\ 
    \psi_i=(\chi_i|_{M^{a,b}_{\mathcal{A}}})^*(\psi)\in H^{a}(M^{a,b}(\mathcal{A})).
\end{gather*}
We have $\omega_i\psi_i^j=0$ for any $i \in E$ and any $j \leq a$ because $\omega\psi^j=0$ in $H^*(G;\mathbb{Z})$.
In particular, $\omega_iz =0$ for any $z \in \ker \left( H^*(G^r;\mathbb{Z}) \to H^*(H_i;\mathbb{Z}) \right)$.
Note that, when $b=1$, the classes $\psi_i$'s have degree $a=a+b-1$, the same degree of the elements $\omega_i$'s.

\section{Central and unimodular arrangements}\label{Sec:3}

Consider a proper map $f\colon N \to M$ between oriented manifolds. 
Such a map induces a pushforward map in cohomology $f_*$ obtained from the pushforward in Borel-Moore homology (for a general reference see \cite{HughesRanickiBook}) by composing with the Poincaré duality isomorphism.
\begin{center}
    \begin{tikzcd}
H^k(N) \arrow[d, "\PD_N"'] \arrow[r, "f_*"] & H^{\dim M - \dim N + k}(M)               \\
H^{\BM}_{\dim N - k}(N) \arrow[r, "f_*^{\BM}"']   & H^{\BM}_{\dim N - k}(M) \arrow[u, "\PD^{-1}_M"']
\end{tikzcd}
\end{center}
The map $f_*$ increases the degree by $\dim M- \dim N$ and has the following properties:
\begin{enumerate}[label=(P\arabic*)]
    \item (Functoriality) $f_* \circ g_*=(f \circ g)_*$.
    \item (Projection formula) $f_*( f^*(y) \smile x) = y \smile f_*(x) $ for any $x \in H^*(N)$ and any $y \in H^*(M)$.
    \item \label{item:natrurality_pushforward} (Naturality) For any pullback diagram
    \begin{center}
        \begin{tikzcd}
N' \arrow[d, "f'"'] \arrow[r, "h"] & N \arrow[d, "f"] \\
M' \arrow[r, "g"']                 & M               
\end{tikzcd}
    \end{center}
    with $f$ and $f'$ proper maps and $\dim M - \dim N = \dim M'- \dim N'$, we have
    \[ g^* \circ f_* = f'_* \circ h^*.\]
    \item \label{item:prop_4_pushforward} (Embedding) If $f$ is a closed embedding then $f_*$ is the composition of the Thom isomorphism for the normal bundle
    \[ f_* \colon H^*(N) \xrightarrow{\operatorname{Th}} H^{*+d}(T, T \setminus N) \to H^{*+d}(M, M \setminus N) \to H^{*+d}(M) \]
    where we identify the normal bundle with a tubular neighborhood  $N \subseteq T \subseteq M$ and $d=\dim M - \dim N$.
    The second map is provided by excision theorem of $M\setminus T$ and the last one is the map from the long exact sequence of the pair $(M,M \setminus N)$.
\end{enumerate}

We consider the pushforward in cohomology $i_*\colon H^0(\mathbb{R}\setminus\{0\})\rightarrow H^{a+b-1}(G\setminus\{e\})$ induced by the closed inclusion $i\colon\mathbb{R}\setminus\{0\}\rightarrow G\setminus\{e\}$, $i(x)=(x,0,\cdots,0,1,\cdots,1)$.
Let $w^+$ and $w^-$ be the two standard generators of $H^0(\mathbb{R}\setminus \{0\})$.
Poincaré Duality 
maps $w^+$ and $w^-$ to the classes in $H^{\BM}_1(\mathbb{R}\setminus\{0\})$ represented by the infinite chains $c^+$ and $c^-$, corresponding respectively to the positive and negative semi-axes.
By checking intersection numbers, we obtain $\PD(\omega)= i_*c^+$ and, when $b=1$, $\PD(\psi)= i_*c^+ + i_*c^- =i_*1$. Note that, if $b>1$, $i_*c^+ + i_*c^-=0$ in $H^{\BM}_1(G\setminus\{e\})$.
Summarizing:
\begin{equation}
\label{eq:pushforward_i}
i_*(w^+)=\omega, \quad i_*(w^-)=
\begin{cases}\psi-\omega &\text{if }b=1,\\
-\omega &\text{otherwise}
\end{cases},
\quad
i_*(1)= \begin{cases}\psi &\text{if }b=1,\\
0 &\text{otherwise.}
\end{cases}    
\end{equation}

\begin{figure}
    \centering
\begin{tikzpicture}
    \draw[black, thin] (-4, -2) rectangle (4,2);
    \draw[postaction={decorate}, thick,decoration={
    markings,
    mark=at position 0.5 with {\arrow{>}}}] (0,0) -- (2,0);
	\draw[postaction={decorate}, thick,decoration={
    markings,
    mark=at position 0.5 with {\arrow{>}}}] (2,0) -- (4,0);
    \node [cross out,draw,minimum size=2pt] at (0,0){};
    \node [cross out,draw,minimum size=2pt] at (2,0){};
     \draw[red, very thick,
        decoration={markings, mark=at position 0.125 with {\arrow{>}}},
        postaction={decorate}]
        (0,0) circle (0.7);
    \draw[blue, very thick,
        decoration={markings, mark=at position 0.125 with {\arrow{>}}},
        postaction={decorate}]
        (2,0) circle (0.5);
    \node at (0,1) {\textcolor{red}{$Y$}};
    \node at (2,1) {\textcolor{blue}{$W$}};
    \node at (1.1,-0.3) {$i_{*}c^{-}$};
    \node at (3.2,-0.3) {$i_{*}c^{+}$};
    \node at (-2.5,1.3) {$(\mathbb{R}\times S^{1})\setminus(0,1)$};    
\end{tikzpicture}
    \caption{In the case $(a,b)=(1,1)$, the homology class $\color{red}Y$ is represented in red and $\color{blue}W$ in blue. The two arrows represent the support of Borel Moore homology classes $i_* c^+, i_*c^- \in H_1^{\BM}(\mathbb{R} \setminus \{0\})$.}
    \label{fig:pushforward}
\end{figure}
%

Now we consider $\mathcal{A}$ as a real hyperplane arrangement, for any $i\in E$ let $w_i^+=\chi_i^{*}(w^+)$ and $w_i^-=\chi_i^{*}(w^-)$ be the classes in $H^0(M^{\mathbb{R}}(\mathcal{A}))$ corresponding to the positive and negative half spaces.

\begin{theorem}[{\cite[Theorem 5]{VarchenkoGelfand87}}]
Let $\AAA$ be a central real hyperplane arrangement.
The ring $H^0(M^{\mathbb{R}}(\mathcal{A}))$ is generated by the classes $w_i^+, w_i^-$ with $i\in E$ and subject to the following relations:
\begin{itemize}
    \item $w_i^-=1-w_i^+$; 
    \item $w_i^+w_i^-=0$;
    \item for any oriented circuit $C = C^{+} \sqcup C^{-}$, $\prod_{i\in C^{+}} w_i^+ \prod_{j\in C^{-}} w_j^-=0$.
\end{itemize}
\end{theorem}
We will write, for $I\subseteq E$, $\omega_I^+= \prod_{i \in I}\omega_i^+$ and $\omega_I^-\prod_{i \in I}\omega_i^-$.
Similarly, $\omega_I=\omega_{i_1}\omega_{i_2} \dots \omega_{i_k}$ and $\psi_I=\psi_{i_1}\psi_{i_2} \dots \psi_{i_k}$ where $I=\{i_1,i_2, \dots, i_k\}$ and the product is taken in the order induced by the ordered ground-set $E$.
\begin{remark}
    From the first and second relations it follows that $(w_i^\pm)^2=w_i^\pm$. 
    Consider a central circuit $C$,
    using the first and the third relations for $C$, we get
    \begin{align}
    0=\prod_{i\in C^{+}} w_i^+ \prod_{j\in C^{-}} (1-w_j^+) = \sum_{J\subseteq C^{-}} (-1)^{|J|} w_{C^+}^+ w^+_J=\sum_{J\subseteq C^{-}}(-1)^{|J|}w^+_{C^+\sqcup J}
    \end{align}
    and for the opposite circuit
    \begin{align}
    0=\prod_{i\in C^{-}} w_i^+ \prod_{j\in C^{+}} (1-w_j^+) = \sum_{J\subseteq C^{+}} (-1)^{|J|} w_{J}^+ w^+_{C^-}=\sum_{J\subseteq C^{+}}(-1)^{|J|}w^+_{J\sqcup C^-}.
    \end{align}
    By summing the equations above with a suitable sign, we get the following relations:
    \begin{align*}
        \sum_{J\subsetneq C^{-}}(-1)^{|J|}w^+_{C^+\sqcup J}-(-1)^{|C|}\sum_{J\subsetneq C^{+}}(-1)^{|J|}w^+_{J\sqcup C^-}=0,
    \end{align*}
    \begin{align}\label{rel}
        \sum_{\substack{K\subseteq C^{-}\\ K\neq \emptyset}}(-1)^{|K|}w^+_{C\setminus K}-\sum_{\substack{K\subseteq C^{+}\\ K\neq \emptyset}}(-1)^{|K|}w^+_{C\setminus K}=0.
    \end{align}
    From this identities, we will obtain a relation in $H^{*}(M^{a,b}(\mathcal{A}))$ by applying the pushforward map $i_*$.
\end{remark}

\begin{definition}
    \label{def:sgn_ell}
    Let $A,B$ be disjoint subsets of $E$. We denote by $\ell(A,B)$ the sign of the permutation taking the $A\sqcup B$ as a ordered subset of $E$ of $\mathcal{A}$ to the concatenation of $A$ and $B$. 
\end{definition}

\begin{definition}
Let $A$ and $B$ be disjoint subsets of $E$ and define
    \[\eta_{A,B}:=(-1)^{d\ell(A,B)}\omega_A\psi_B \in H^{*}(M^{a,b}(\mathcal{A})).
    \]
\end{definition}


\begin{remark}
    \label{rmk:independece_of_subarr}
    Let $\mathcal{B} \subset \AAA$ be a sub-arrangement and $j\colon M^{a,b}_\AAA \hookrightarrow M^{a,b}_\mathcal{B}$ be the inclusion of the complements.
    From the definitions it follows that that $j^*\eta^{\mathcal{B}}_{C,D}= \eta_{C,D}^\AAA$.
    So we will not specify the ambient space.
    Since $\omega_A z =0$ for any $z \in \ker \left( H^*(G^r;\mathbb{Z}) \to H^*(\cap_{i \in C}H_i;\mathbb{Z}) \right)$, we have $\eta_{C,D}=0$ if $D$ is dependent on $C$.
\end{remark}

\begin{lemma}\label{lemma1}
    Let $\mathcal{A}$ be a unimodular central arrangement of rank $r$. Consider the closed immersion $i\colon M^{\mathbb{R}}(\mathcal{A})\rightarrow M^{a,b}(\mathcal{A})$ and its pushforward in cohomology $i_*\colon H^0(M^{\mathbb{R}}(\mathcal{A}))\rightarrow H^{r(a+b-1)}(M^{a,b}(\mathcal{A}))$. Let $I\subseteq E$ be an independent set and $B\subseteq E$ a basis containing $I$:
    \begin{itemize}
        \item If $b=1$, then  $i_*(w_I^+) = \sgn(B)^d \eta_{I,B\setminus I}$.
        \item If $b>1$, then  $i_*(w_I^+)= \begin{cases}
            0 &\text{ if } \rk(I)\neq r, \\
            \sgn(I)^d \omega_I &\text{ otherwise.}
        \end{cases}$
    \end{itemize}
\end{lemma}

\begin{proof}
Let $\mathcal{B} \subset \AAA$ be the sub-arrangement $\{ H_b \mid b \in B \}$.
Consider the inclusions $j^{\mathbb{R}} \colon M^{\mathbb{R}}(\mathcal{A}) \hookrightarrow M^{\mathbb{R}}(\mathcal{B})$ and $j \colon M^{a,b}(\mathcal{A}) \hookrightarrow M^{a,b}(\mathcal{B})$.
Recall that $d=a+b-1$ is the codimension of $\RR$ in $G$.
The following diagram commutes by the naturality of pushforward in cohomology (property \ref{item:natrurality_pushforward}):
\[
\begin{tikzcd}
H^0(M^{\mathbb{R}}(\mathcal{A})) \arrow[d, "i_{\mathcal{A *}}"'] & H^0(M^{\mathbb{R}}(\mathcal{B})) \arrow[l, "{j^{\mathbb{R}}}^*"'] \arrow[d, "i_{\mathcal{B} *}"]& H^0(\mathbb{R}\setminus 0)^{\otimes r} \arrow[d, "i_*^{\otimes r}"]  \arrow[l, "\simeq", "k^{\RR}"']\\
{H^{rd}(M^{a,b}(\mathcal{A}))}                             & {H^{rd}(M^{a,b}(\mathcal{B}))} \arrow[l, "{j^{a,b}}^*"]                  & H^{d}(G\setminus e)^{\otimes r} \arrow[l, "\simeq"', "k^{a,b}"]
\end{tikzcd}\]
The horizontal maps in the left-hand square are induced by the inclusion of the complement of $\mathcal{A}$ in the complement of $\mathcal{B}$, in $\mathbb{R}^r$ and $G^r$ respectively. Since $M^{\mathbb{R}}(\mathcal{B}) \cong (\mathbb{R}\setminus \{0\})^r$ and $M^{a,b}(\mathcal{B}) \cong (G\setminus \{e\})^r$, the horizontal maps in the right-hand square are given by the K\"unneth isomorphisms.
Note that the K\"unneth isomorphism $k^{a,b}$ depends on the sign of the basis $B$.

For the sake of notation we assume that $I$ are the first elements of $B$, the general case differs by the sign $(-1)^{d\ell(I,B\setminus I)}$ (where $\ell(A,B)$ is introduced in Definition \ref{def:sgn_ell}) given by the reordering of $\omega_i$ and $\psi_b$ for $i \in I$ and $b\in B \setminus I$.
We have
\begin{align*}
    i_{\AAA*}(\omega_I^+)= i_{\AAA*}j^{\RR *}k^\RR((w^+)^{\otimes|I|}\otimes1^{\otimes n-|I|}) = j^{a,b *}k^{a,b}((i_* w^+)^{\otimes|I|}\otimes i_*1^{\otimes n-|I|})
\end{align*}
We use eq.\ \eqref{eq:pushforward_i} distinguishing two cases.
If $b=1$, 
\[i_{\AAA*}(\omega_I^+)=j^{a,b *}k^{a,b} (\omega^{\otimes|I|}\otimes \psi^{\otimes n-|I|})= j^{a,b *}(\sgn(B)^d \omega_I \psi_{B \setminus I})= \sgn(B)^d \eta_{I, B\setminus I}, \]
where in the last equality we use \Cref{rmk:independece_of_subarr}.
Otherwise $b>1$, $i_*(1)=0$ and hence ${i_{\AAA}}_*(w^+_I)=0$ if $\rk(I)\neq r$.
If $\rk(I)=r$, then
\[i_{\AAA*}(\omega_I^+)= j^{a,b *}k^{a,b} (\omega^{\otimes n})= j^{a,b *}(\sgn(I)^d \omega_I)= \sgn(I)^d \omega_I. \]
This completes the proof.
\end{proof}

\begin{remark}
    Note that, also in the case $b=1$, the image of $i_*$ does not depend on the choice of the (unimodular) basis $B$. In fact, if $B'\subseteq E$ is another (unimodular) basis containing $I$, $\sgn (B)^d\psi_{B\setminus I} \omega_I = \sgn (B')^d\psi_{B'\setminus I} \omega_I$ by using the linear relations among the classes $\psi$.
\end{remark}

\begin{definition}
    Let $C\subseteq E$ be a circuit and $B\supseteq C$ such that $B\setminus{\{i\}}$ is a basis for any $i\in C$. We denote by $c_i$ the sign of the determinant of the matrix whose columns are the elements in $B\setminus{\{i\}}$.
\end{definition}
    Up to a global sign, the $c_i$'s do not depend on the choice of $B$. 

\begin{lemma}\label{lemma2}
    Let $\mathcal{A}$ be a central arrangement and $C\subseteq E$ an unimodular oriented circuit with $C=C^+\sqcup C^-$. The following relations holds in $H^*(M^{a,b}(\mathcal{A}))$:
    \begin{itemize}
        \item If $b=1$, 
        \begin{equation}
        \sum_{\substack{K\subseteq C^- \\ K\neq \emptyset}}(-1)^{|K|}c_{i_K}^d\eta_{C\setminus K,K\setminus i_K}-\sum_{\substack{K\subseteq C^+ \\ K\neq \emptyset}}(-1)^{|K|}c_{i_K}^d\eta_{C\setminus K, K\setminus i_K}=0
\end{equation}
where $i_K\in K$.
\item If $b>1$ and $d$ odd,
    \begin{equation}
    \sum_{i\in C} (-1)^{\lvert C_{<i} \rvert} \omega_{C\setminus\{i\}}=0.
    \end{equation}
    
\item If $b>1$ and $d$ even,
    \begin{equation}
    \sum_{i\in C^-}  \omega_{C\setminus\{i\}} - \sum_{i\in C^+}  \omega_{C\setminus\{i\}}=0.
    \end{equation}
    \end{itemize}
\end{lemma}
\begin{proof} Let $r$ be the rank of the arrangement $\mathcal{A}$ and $n$ be the rank of the circuit $C$.
Consider the additive map given by the following composition:
\[
H^0(M^{\mathbb{R}}(C)) \xlongrightarrow{i_*} 
{H^{n(a+b-1)}(M^{a,b}(C))} \xlongrightarrow{p^*} H^{n(a+b-1)}(M^{a,b}(\mathcal{A}))
\]
where $p$ is the restriction of the projection $G^r \to G^r/ \cap_{c \in C} H_c \simeq G^n $.

Consider the relation \eqref{rel} in $H^0(M^{\mathbb{R}}(C))$ and we apply Lemma \ref{lemma1} to the essentialization of the arrangement $C$.
For any independent $J\subsetneq C^+$ (resp. $J\subsetneq C^-$), the choice of $i_K\in K=C^+\setminus J$ (resp. $i_K\in K=C^- \setminus J$) corresponds to complete $C^+\bigsqcup J=C\setminus K$ to a basis $C\setminus i_K$ of the sublattice generated by the elements of $C$.
By Lemma \ref{lemma1} and \Cref{rmk:independece_of_subarr} we have:
 \[ p^*(i_*(\omega^+_{C \setminus K}))= p^*(c_{i_K}^d\eta_{C\setminus K,K\setminus i_K}) = c_{i_K}^d\eta_{C\setminus K,K\setminus i_K},\]
and so Equation (\ref{rel}) is send to
$$
\sum_{\substack{K\subseteq C^- \\ K\neq \emptyset}}(-1)^{|K|}c_{i_K}^d\eta_{C\setminus K,K\setminus i_K}-\sum_{\substack{K\subseteq C^+ \\ K\neq \emptyset}}(-1)^{|K|}c_{i_K}^d \eta_{C\setminus K, K\setminus i_K}=0
$$
for $b=1$, and to
$$
\sum_{i \in C^- } c_i^d \omega_{C\setminus i} - \sum_{i \in C^+} c_i^d \omega_{C\setminus i}=0
$$
for $b>1$.
Since $C$ is unimodular, the dependence relation among the characters of the circuit is $\sum_{i\in C} (-1)^{\lvert C_{<i} \rvert} c_i\chi_i=0$, hence $i \in C^+$ if and only if $c_i = (-1)^{\lvert C_{<i} \rvert}$.
\end{proof}

\begin{remark}
    Lemma \ref{lemma2} holds also for a central circuit $C$ in a non central arrangement $\mathcal{A}$.  It is sufficient to consider the map $H^*(M^{a,b}(C))\rightarrow H^*(M^{a,b}(\mathcal{A}))$ induced by the inclusion $M^{a,b}(\mathcal{A})\hookrightarrow M^{a,b}(C)$.
\end{remark}

\begin{example}\label{ex: CU}
\begin{figure}
    \centering
\begin{tikzpicture}
 \draw[black, thick] (-2,0) -- (2,0);
    \draw[black, thick] (0,-2) -- (0,2);
    \draw[black, thick] (1.8,-1.8) -- (-1.8,1.8);
    \draw[->] (1.8,0) -- (1.8,0.3);
    \node at (1.8,-0.3){$H_2$};
    \draw[->] (0,1.8) -- (0.3,1.8);
    \node at (-0.3,1.8){$H_1$};
    \draw[->] (-1.6,1.6) -- (-1.4,1.8); 
    \node at (-1.8,1.4){$H_3$};
\end{tikzpicture}
\quad \quad \quad \quad
\begin{tikzpicture}[scale=1.3]
\draw[cyan, thick] (-1.5,-1.5) -- (-1.5,1.5);
\draw[black, thick] (1.5,-1.5) -- (1.5,1.5);
\draw[black, thick] (-1.5,-1.5) -- (1.5,-1.5);
\draw[black, thick] (-1.5,1.5) -- (1.5,1.5);
\draw[red, thick] (-1.5,1.5) -- (1.5,-1.5);
\draw[green, thick] (-1.5,-1.5) -- (1.5,-1.5);
\end{tikzpicture}
    \caption{A picture of the arrangement in Example \ref{ex: CU} for $(a,b)=(1,0)$ on the left, and $(a,b)=(0,1)$ on the right.}
    \label{fig:Example1}
\end{figure}
    Let us consider the central arrangement $\mathcal{A}=\{H_i\}_{i=1,2,3}$ in $G^2$ defined by the columns of the matrix
    $$\begin{pmatrix}
1 & 0 & 1 \\
0 & 1 & 1 
\end{pmatrix}$$
(see Figure \ref{fig:Example1}). This is an arrangement of type $A_3$, hence $M(\mathcal{A})$ is the configuration space of $3$ points in $G$ (c.f.\ \Cref{Sec: conf}).
The subspace $H^*(G^2) \subset H^*(M(\mathcal{A}))$ is generated by the classes $\psi_1^j,\psi_2^j$ and $\psi_3^j=\psi_1^j+\psi_2^j$, with $j=1,\cdots, a$.
Furthermore, we have defined the classes $\omega_i\in H^{a+b-1}(M(\mathcal{A}))$ for $i=1,2,3$.
The circuit $C=C^+\sqcup C^-=\{1,2\}\sqcup \{3\}$ gives the following relations
\begin{itemize}
    \item if $b=1$, \begin{equation}\label{CU1}
        \omega_1\omega_2- (-1)^d \omega_2\omega_3-\omega_1\omega_3-\omega_3\psi_1=0,
    \end{equation}
    \item if $b>1$, \begin{equation}\label{CU2}
        \omega_1\omega_2-(-1)^d\omega_2\omega_3-\omega_1\omega_3=0.
    \end{equation}
\end{itemize}
\end{example}

\section{The long exact sequence of the pair}\label{Sec:4}
Now, we deal with the case of general arrangements, possibly non-central and non-unimodular.
Firstly, we prove a variation of the Brieskorn Lemma, a deletion-restriction short exact sequence, and an identity between Poincaré and characteristic polynomial.
We will apply these results in \Cref{Sec:5} to unimodular covers.

Given $H\in \AAA$ a subvariety of the arrangement, we denote by $\AAA':= \AAA \setminus \{H\}$ the \emph{deletion} of $H$ and $\AAA''$ the \emph{restriction} to $H$, i.e.\ the arrangement of connected components of $K\cap H$ for $K \in \AAA'$.
The complement $M(\AAA')=M(\AAA) \sqcup M(\AAA'')$ is a  disjoint union of an open and a closed subset.
The relative cohomology $H^*(M(\AAA'),M(\AAA))$ is isomorphic to $H^{*-a-b}(M(\AAA''))$ via the Thom isomorphism and excision theorem.
Therefore, the forgetful map of the long exact sequence of the pair $(M(\AAA'),M(\AAA))$ can be identified through the aforementioned isomorphism with the pushforward $\iota_*$, where $\iota \colon M(\AAA'') \to M(\AAA')$ is the closed immersion (see property~\ref{item:prop_4_pushforward}).

For any layer $p \in \LL(\AAA)$ we consider the \emph{local arrangement} $\AAA_p:= \{H \in \AAA \mid p \subseteq H\}$ at $p$ and let $g_p \colon M(\AAA) \to M(\AAA_p)$ be the inclusion of the complements.

\begin{theorem}\label{thm: nonC}
Let $\mathcal{A}$ be an arrangement of rank $r$ and $H \in \mathcal{A}$ a subvariety, then:
\begin{enumerate}
\item $\iota_* \colon H^*(M(\mathcal{A}'')) \to H^{*+a+b}(M(\mathcal{A}'))$ is the zero map,
\item the map $ \oplus_{p \in \mathcal{L}^r} g_p^* \colon \bigoplus_{p \in \mathcal{L}^r} H^*(M(\mathcal{A}_p)) \to H^*(M(\mathcal{A}))$ is surjective.
\end{enumerate}
\end{theorem}
\begin{proof}
In the central case the theorem follows from \cite[Theorem 7.7]{LiuTranYoshinaga2021}. 
We prove the claims by induction on the rank and the cardinality of the arrangement $\mathcal{A}$.
The base case is a central arrangement and the results hold.
The assumption that $\mathcal{A}$ is not central implies $\lvert \mathcal{A}_p \rvert < \lvert \mathcal{A} \rvert$ for any $p \in \mathcal{L}^r$.
Consider $p\in \LL(\AAA'')$ and let $\tilde{p} \in \LL(\AAA')$ be the unique minimal layer containing $p$ (if $H$ is not a coloop then $p=\tilde{p}$), the diagram
\begin{center}
\begin{tikzcd}
M(\mathcal{A}'') \arrow[r, "\iota"] \arrow[d, "g''_p"'] & M(\mathcal{A}') \arrow[d, "g'_{\tilde{p}}"] \\
M(\mathcal{A}''_p) \arrow[r, "\iota_p"']                      & M(\mathcal{A}'_{\tilde{p}})                      
\end{tikzcd}
\end{center}
is a pullback diagram and so $\iota_* \circ g^{\prime \prime *}_p= g^{\prime *}_{\tilde{p}} \circ \iota_{p *}$ by property \ref{item:natrurality_pushforward}.
Since the long exact sequence of the pair and the Thom isomorphism are functorial, the following diagram commutes:
\begin{center}
    \begin{tikzcd}
H^{k-d-1}(M(\mathcal{A}_p'')) \arrow[r, "\iota_p^*"] \arrow[d, "g_p^{\prime \prime *}"] & H^k(M(\mathcal{A}'_{\tilde{p}})) \arrow[r, "j_p^*"] \arrow[d, "g^{\prime *}_{\tilde{p}}"] & H^k(M(\mathcal{A}_p)) \arrow[r, "\delta_p"] \arrow[d, "g^*_p"] & H^{k-d}(M(\mathcal{A}_p'')) \arrow[d, "g_p^{\prime \prime *}"]  \\
H^{k-d-1}(M(\mathcal{A}'')) \arrow[r,  "\iota_* "] & H^k(M(\mathcal{A}')) \arrow[r, "j^*"]         & H^k(M(\mathcal{A})) \arrow[r,  "\delta "]                      & H^{k-d}(M(\mathcal{A}'')) 
\end{tikzcd}
\end{center}
For $p \in \LL(\AAA) \setminus \LL(\AAA'')$ we have a similar diagram
\begin{center}
    \begin{tikzcd}
0 \arrow[r] \arrow[d] & H^k(M(\mathcal{A}'_{p})) \arrow[r, "j_p^*"] \arrow[d, "g^{\prime *}_{p}"] & H^k(M(\mathcal{A}_p)) \arrow[r] \arrow[d, "g^*_p"] & 0 \arrow[d]  \\
H^{k-d-1}(M(\mathcal{A}'')) \arrow[r,  "\iota_* "] & H^k(M(\mathcal{A}')) \arrow[r, "j^*"]         & H^k(M(\mathcal{A})) \arrow[r,  "\delta "]                      & H^{k-d}(M(\mathcal{A}'')) 
\end{tikzcd}
\end{center}
because $\AAA'_p=\AAA_p$.
We consider the direct sum of previous sequences for all $p \in \LL(\AAA)$.
Since $\rk( \mathcal{A}'' ) < \rk (\mathcal{A})$, the map $\oplus_{p \in \mathcal{L}^r} g_p^{\prime \prime *}$ is surjective and $\iota_{p *}=0$ by inductive step:
\begin{center}
\begin{tikzcd}
    \oplus_{p}H^{k-d-1}(M(\mathcal{A}_p'')) \arrow[r, "0"] \arrow[d, "\oplus_{p}g_p^{\prime \prime *}", two heads] & \oplus_{\tilde{p}}H^k(M(\mathcal{A}'_{\tilde{p}})) \arrow[r, "\oplus_{p}j_p^*"] \arrow[d, "\oplus_{p}g^{\prime *}_{\tilde{p}}", two heads] & \oplus_{p}H^k(M(\mathcal{A}_p)) \arrow[r, "\oplus_{p}\delta_p"] \arrow[d, "\oplus_{p}g^*_p"] & \oplus_{p}H^{k-d}(M(\mathcal{A}_p'')) \arrow[d, "\oplus_{p}g_p^{\prime \prime *}", two heads] \\
    H^{k-d-1} (M(\mathcal{A}'')) \arrow[r,  "\iota_* "] & H^k(M(\mathcal{A}')) \arrow[r, "j^*"] & H^k(M(\mathcal{A})) \arrow[r,  "\delta "] & H^{k-d} (M(\mathcal{A}''))
\end{tikzcd}
\end{center}
It follows that $\iota_*=0$.
The two rows are exact because they are long exact sequences of pairs and by diagram chasing the map $ \oplus_{p \in \mathcal{L}^r} g_p^*$ is surjective.
\end{proof}

The corollary below follows directly from \Cref{thm: nonC}.
\begin{corollary}\label{cor: ses}
Let $\mathcal{A}$ be an arrangement of rank $r$ 
and $H \in \mathcal{A}$ a subvariety. The following sequence is exact
\[0 \to H^*(M(\mathcal{A}')) \xrightarrow{j^*} H^*(M(\mathcal{A})) \xrightarrow{\res} H^{*-d}(M(\mathcal{A}'')) \to 0. \]
\end{corollary}

The characteristic polynomial of an abelian arrangement is 
\[\chi_\AAA(t):= \sum_{W \in \mathcal{L}(\AAA)} \mu(\hat{0},W) t^{r-\rk(W)},\]
where $\mu$ is the M\"obius function of the poset of layers $\mathcal{L}(\AAA)$.

\begin{corollary}\label{cor: gen}
    For any abelian arrangement $\mathcal{A}$ in a non-compact group $G$, the 
     Poincaré polynomial only depends on the characteristic polynomial of the arrangement:
    \[P_{\mathcal{A}}(t)=(-t^{a+b-1})^{r}\chi_{\mathcal{A}}\Big(-\frac{(1+t)^a}{t^{a+b-1}}\Big). \]
\end{corollary}
\begin{proof}
    The identity follows by deletion-contraction recurrence, using the exact sequence in Corollary \ref{cor: ses}.
\end{proof}

    The same formula for the Poincaré polynomial is proven for central abelian arrangements in \cite[Theorem 7.8]{LiuTranYoshinaga2021}

\begin{corollary}
\label{cor:omega_C_nullo}
    Let $C$ be a unimodular, possibly non-central, circuit of $\mathcal{A}$. Then
    \[\omega_C=0.\]
\end{corollary}
\begin{proof}
    Assume that $C=\{1,2,\dots,n\}$ and consider the subvariety $H_0$ obtained as the unique translated of $H_1$ containing $\cap_{i \in C\setminus\{1\}} H_i$.
    The set $D:=(C \cup \{0\}) \setminus \{1\}$ is a central circuit  and we can multiply the relations of Lemma \ref{lemma2} by $\omega_1$ obtaining
    \[ \sum_{\substack{K\subseteq D^- \\ K\neq \emptyset}}(-1)^{|K|}c_{i_K}^d\omega_1\eta_{D\setminus K,K\setminus i_K}-\sum_{\substack{K\subseteq D^+ \\ K\neq \emptyset}}(-1)^{|K|}c_{i_K}^d\omega_1\eta_{D\setminus K, K\setminus i_K}=0.\]
    The elements $0,1 $ are parallel, so $\omega_0\omega_1=0$ and $\omega_1\psi_0=0$, see also Remark \ref{rmk:independece_of_subarr}.
    In particular, all the terms in the above sum are zero unless the one with $K=\{0\}$, hence the equation reduces to $\omega_C =0$ in the cohomology ring of the arrangement $\AAA \cup \{H_0\}$.
    The claimed equality follows from the injectivity of the map $H^*(M(\AAA)) \to H^*(M(\AAA \cup \{H_0\}))$ by Corollary \ref{cor: ses}.
\end{proof}

\begin{example}\label{ex: NCU}

\begin{figure}
    \centering
\begin{tikzpicture}[scale=0.5]
    \draw[black, thick] (-5,0) -- (2,0);
    \draw[black, thick] (0,-2) -- (0,6);
    \draw[black, thick] (-5,4) -- (2,4);
    \draw[black, thick] (2,-2) -- (-5,5);
    \draw[black, thick] (-5,1) -- (-2,-2);
    \draw[black, thick] (-2,6) -- (2,2);
    \draw[->] (0,5.8) -- (0.4,5.8); {}
    \draw[->] (1.8,0) -- (1.8,0.4); {}
    \draw[->] (1.8,4) -- (1.8,4.4); {}
    \draw[->] (-4.5,0.5) -- (-4.2,0.8); {}
    \draw[->] (-2,2) -- (-1.7,2.3); {}
    \node at (-4.2,1.1){$H_{3^{'}}$};
    \node at (-4.4,5.2){$H_3$};
    \node at (1.8,-0.5){$H_2$};
    \node at (1.8,3.5){$H_{2^{'}}$};
    \node at (-0.5,-1.5){$H_1$};
\end{tikzpicture}
\quad \quad \quad \quad
\begin{tikzpicture}[scale=1.3]
\draw[cyan, thick] (-1.5,-1.5) -- (-1.5,1.5);
\draw[black, thick] (1.5,-1.5) -- (1.5,1.5);
\draw[black, thick] (-1.5,-1.5) -- (1.5,-1.5);
\draw[black, thick] (-1.5,1.5) -- (1.5,1.5);
\draw[red, thick] (-1.5,1.5) -- (1.5,-1.5);
\draw[green, thick] (-1.5,-1.5) -- (1.5,-1.5);
\draw[violet, thick] (-1.5,0) -- (1.5,0);
\draw[orange, thick] (-1.5,0) -- (0,-1.5);
\draw[orange, thick] (1.5,0) -- (0,1.5);
\end{tikzpicture}
    \caption{A picture of the non central arrangement in Example \ref{ex: NCU} when $(a,b)=(1,0)$ on the left, and $(a,b)=(0,1)$ on the right.}
    \label{fig:Example2}
\end{figure}

   Let us consider the arrangement obtained from the arrangement
   in the Example \ref{ex: CU} by adding the two hyperplanes $H_{2'}=\{\chi_2^{-1}(-e)\}$ and $H_{3'}= \{(\chi_1+\chi_2)^{-1}(-e)\}$ as shown in Figure \ref{fig:Example2}. \\
   The arrangement is not central anymore, but it is still unimodular. As in the previous example, the cohomology $H^*(M(\mathcal{A}))$ is generated by the classes 
   $\psi_1^j, \psi_2^j, \psi_3^j=\psi_1^j+\psi_2^j$, $j=1, \cdots, a$ and 
   $\omega_1, \omega_2, \omega_{2'}, \omega_3, \omega_{3'}$.
   There are two central circuits:
   $$C=C^+ \sqcup C^- = \{1,2\} \sqcup \{3\}, \quad C'=(C')^+ \sqcup (C')^- = \{1,2'\} \sqcup \{3'\}$$
   When $b=1$, the first circuit leads to the same relation eq.\ \eqref{CU1} of \Cref{ex: CU}, while the second one to the new relation:
    \begin{equation}\label{NCU}
        \omega_1\omega_{2'} + (-1)^d\omega_{2'}\omega_{3'} - \omega_1\omega_{3'} - \omega_{3'}\psi_1 = 0.
    \end{equation}
    The case $b>1$ is analogous.
    The Poincaré polynomial of the complement is given by $P_\AAA(t)=(1+t)^{2a}+5(1+t)^a+6$.
    The four non-central circuits provide the following relations:
    \begin{align*}
        &\omega_2 \omega_{2'} =0, && \omega_3 \omega_{3'} =0, \\
        &\omega_1 \omega_2 \omega_{3'} =0, && \omega_1 \omega_{2'} \omega_{3} =0,
    \end{align*}
    by Corollary \ref{cor:omega_C_nullo}.
\end{example}

\section{Non-unimodular case}\label{Sec:5}

In this section we prove our main result on the cohomology of the complement of abelian arrangements, without any assumption of centrality and unimodularity.
We adapt the argument in \cite{CDDMP2020} constructing an separating cover for an abelian arrangement $\mathcal{A}$.

\begin{definition}
    A \emph{separating cover} of $A \subseteq E$ is a cover $\pi \colon  \Hom(\Gamma, G) \to \Hom(\Lambda, G)$ such that the connected components of $\pi^{-1}(H_i)$ for $i \in A$ form a unimodular, possibly non-central, arrangement.
\end{definition}
A covering is induced from an extension $\Gamma \supset \Lambda$, for each character $\chi \in \Lambda$ let $a\in \mathbb{N}$ be biggest integer such that $\frac{\chi}{a} \in \Gamma$.
The cover separates $A$ if and only if the characters $\frac{\chi_i}{a_i}$ generate a split subgroup of $\Gamma$.
In general, there is no separating cover for $A$, but if $A$ has nullity at most one then it exists.

Indeed, let $X\subseteq E$ of nullity $1$, hence $X=C\sqcup F$ contains a unique circuit $C$ and set $n=\rk(C)$.
Let
\begin{align*}
    a_i=\begin{cases}m(X) \prod_{j\in C\setminus\{i\}}m(C\setminus\{j\}) &\text{ for } i\in C\\
    m(X) &\text{ for } i\in F\end{cases}
\end{align*}
and let $\Lambda'$ be the lattice in $\Lambda\otimes_{\mathbb{Z}} \mathbb{Q}$ generated by the characters $\frac{\chi_i}{a_i}$.
This cover is separating for $X$ of nullity $1$.

\begin{lemma}
The lattice $\Lambda'$ contains $\Lambda$ and the inclusion induces a covering $\pi \colon \Hom (\Lambda',G)\rightarrow \Hom (\Lambda,G)$ of degree $m(C)^a m(X)^{a(r-1)}\prod_{i\in C} m(C\setminus i)^{a(n-1)}$.
\end{lemma}

\begin{proof}
    The proof of the first statement is analogous to that of \cite[Lemma 6.4]{CDDMP2020}. Let $i\in C$, as in \cite[Lemma 6.5]{CDDMP2020} we compute the index of $\Lambda$ in $\Lambda'$
    \begin{align*}
        [\Lambda': \Lambda]&=\frac{[\Lambda':\Lambda_{X\setminus i}]}{[\Lambda: \Lambda_{X\setminus i}]}=\frac{\prod_{j\in X\setminus i}a_j}{m(X\setminus i)}=\frac{m(C)m(X)^r\prod_{j\in C\setminus i}\prod_{l \in C\setminus j}m(C\setminus l)}{m(X)m(C\setminus i)}\\
        &=m(C)m(X)^{r-1}\prod_{i\in C}m(C\setminus i)^{n-1}
    \end{align*}
    It follows that the degree of the covering is $[\Lambda':\Lambda]^a$.
\end{proof}

Let $\mathcal{A}_U$ be the arrangement in $U:=\Hom (\Lambda',G)$ given by the set of connected components of the preimages of the subvarieties in $\mathcal{A}$
$$\mathcal{A}_U=\bigcup_{H\in \mathcal{A}} \pi_0(\pi^{-1}(H)).$$ 

\begin{lemma}[{\cite[Lemma 6.7]{CDDMP2020}}] The arrangement $\mathcal{A}_U$ is unimodular. 
\end{lemma}

Let $H_i \in \mathcal{A}$, $L$ a connected component of $\pi^{-1}(H_i)$ and $q \in L$. The subvariety $L$ has equation $\frac{\hat{\chi_i}}{a_i}=\frac{\hat{\chi_i}}{a_i}(q)$, where $\hat{\chi}=\chi\circ\pi$. 
We define
\begin{align}
    & \omega^U_i(q)=\bigg( \frac{\hat{\chi_i}}{a_i}-\frac{\hat{\chi_i}}{a_i}(q)\bigg)\bigg|_{M(\mathcal{A}_U)}^*(\omega) \quad \in H^{a+b-1}(M(\mathcal{A}_U)), \\
    & \psi_i^{j,U}=\bigg(\frac{\hat{\chi_i}}{a_i}\bigg)\bigg|_{M(\mathcal{A}_U)}^*(\psi^j) \quad \in H^{1}(M(\mathcal{A}_U)), \\ 
    & \psi_i^U=\bigg(\frac{\hat{\chi_i}}{a_i}\bigg)\bigg|_{M(\mathcal{A}_U)}^*(\psi) \quad \in H^{a}(M(\mathcal{A}_U)).
\end{align}

\begin{remark}
For $H_i \in \mathcal{A}$, let $L$ be a connected component of $\pi^{-1}(H_i)$ and $p,q \in L$. Then we have $\frac{\hat{\chi_i}}{a_i}(q)=\frac{\hat{\chi_i}}{a_i}(p)$ and hence $\omega^U_i(q)=\omega^U_i(p)$.
\end{remark}

As in Section \ref{Sec:3}, for $I \subseteq E$, we write $\omega_I^U=\omega_{i_1}^U\omega_{i_2}^U \dots \omega_{i_k}^U$ and $\psi_I^U=\psi_{i_1}^U\psi_{i_2}^U \dots \psi_{i_k}^U$ where $I=\{i_1,i_2, \dots, i_k\}$ and the product is taken in the order induced by the ordered groundset $E$.

Let $A\subseteq E$ be a independent subset, $W$ a connected component of $\bigcap_{i\in A}H_i$ and $p\in W$. Since $\pi^*$ is injective, we define $\omega_{W,A}$ as the unique class in $ H^*(M(\mathcal{A});\mathbb{Z})$ such that
\begin{align}
    \pi^*(\omega_{W,A})=\frac{1}{|L\cap \pi^{-1}(p)|}\sum_{q\in \pi^{-1}(p)}\omega_A^U(q) = 
    \sum_{\tilde{L} \text{ c.c. of } \pi^{-1}(W)}\omega_A^U(q_{\tilde{L}})
\end{align}
where $L$ is any connected component of $\pi^{-1}(W)$ and $q_{\tilde{L}} \in \tilde{L}\cap\pi^{-1}(p)$ for each connected component $\tilde{L}$ of $\pi^{-1}(W)$.
For any $A,B\subseteq E$ disjoint, we denote
$$\eta^U_{A,B}(q)=(-1)^{d\ell(A,B)}\omega_A^U(q)\psi_B^U\in H^*(M(\mathcal{A}_U))$$
and $$\eta_{W,A,B}=(-1)^{d\ell(A,B)}\omega_{W,A}\psi_B\in H^*(M(\mathcal{A})).$$
The classes $\omega_{W,A}$ do not depend on the choice of the cover $\pi\colon U \to G^r$ as in \cite[Theorem 5.4]{CDDMP2020}.

\begin{lemma}\label{lemma: rel1}
    Let $A$ be an independent set and $W$ a connected component of $\bigcap_{i\in A}H_i$.
    Then, for every $z\in \ker \left( H^*(G^r;\mathbb{Z}) \to H^*(W;\mathbb{Z}) \right)$ 
    \[ \omega_{W,A}z=0.\]
    In particular, if $B\subseteq E$ is dependent on $A$ then for any $j=1,\cdots, a$,
    \begin{equation*}
        \omega_{W,A}\psi^j_B=0.
    \end{equation*}
%
\end{lemma}
\begin{proof}
    If $A$ is a unimodular subset then $\omega_{W,A}= \omega_{a_1} \omega_{a_2} \dots \omega_{a_k}$ and $\omega_{W,A}\psi_i^j=0$ for every $i \in A$ and any $j\leq a$.
    Such elements $\psi_i^j$ generates $\ker \left( H^*(G^r;\mathbb{Z}) \to H^*(W;\mathbb{Z}) \right)$ as an ideal and hence $\omega_{W,A}z=0$.

    In the non unimodular case, we consider the separating cover $\mathcal{A}_U$ and let $L$ be a connected component of $\pi^{-1}(W)$.
    Since $\pi^*$ is injective, it is enough to prove that $\omega^U_A(q_L)z^U=0$ for any $z^U \in \ker \left( H^*(U;\mathbb{Z}) \to H^*(L;\mathbb{Z}) \right)$, because $\pi^*(\omega_{W,A})$ is a linear combination of $\omega^U_A(q_L)$.
    The arrangement $\mathcal{A}_U$ is unimodular and the result follows as in the previous case.

    The last statement follows from the fact that $\psi_B^j \in \ker \left( H^*(G^r;\mathbb{Z}) \to H^*(W;\mathbb{Z}) \right)$ if (and only if) the set $B$ is dependent on $A$.
\end{proof}

\begin{lemma}\label{lemma3} Let $X$ be a set of nullity $1$ and $A,B\subseteq X$ such that $A\sqcup B$ is a maximal independent set. Let $W$ be a connected component of $\bigcap_{i\in A}H_i$ and $p\in W$. Then,
\begin{equation*}
    \pi^*(\eta_{W,A,B})=\frac{ m(A\cup B)^a}{m(A)^a}\sum_{q\in \pi^{-1}(p)}\eta_{A,B}^U(q)
\end{equation*}    
\end{lemma}
\begin{proof}
    By definition, 
    $$\pi^*(\psi_B)=\big(\prod_{i\in B}a_i^a\big)\psi_B^U$$
    The following computation of $|L\cap \pi^{-1}(p)|$ is analogous to that in \cite[Lemma 6.10]{CDDMP2020}.
    The cardinality of the inverse image of $p$ is the degree of the covering. The number of connected components of $\pi^{-1}(W)$ is $$([\Lambda(A):\Lambda^A])^a=\bigg(\frac{\prod_{i\in A}a_i}{m(A)}\bigg)^a$$
    Therefore, 
    \begin{align*}
    |L\cap\pi^{-1}(p)|&=\frac{m(C)^am(X)^{a(r-1)}\prod_{i\in C}m(C\setminus i)^{a(n-1)}m(A)^a}{\prod_{i\in A}a_i^a}\\
    &=\frac{m(C)^a m(X)^{a(r-1)}\prod_{i\in C}m(C\setminus i)^{a(n-1)}m(A)^a}{m(X)^{a|A|}\prod_{i\in A\cap C}\prod_{j\in C\setminus i}m(C\setminus j)^a}\\
    &=m(C)^a m(X)^{a(r-|A|-1)}m(A)^a\prod_{i\in C}m(C\setminus i)^{a(n-|A\cap C| -1)}\prod_{i\in A \cap C}m(C\setminus i )^a 
    \end{align*}
    It follows that
    \begin{align*}
        \pi^*(\eta_{W,A,B})=\frac{\prod_{i\in B}a_i^a}{|L\cap\pi^{-1}(p)|}(-1)^{d\ell(A,B)}\sum_{q\in \pi^{-1}(p)}\omega_A^U(q)\psi_B^U
    \end{align*}
    where,
    \begin{align*}
        \frac{\prod_{i\in B}a_i^a}{|L\cap\pi^{-1}(p)|}&=\frac{m(X)^{a|B|}\prod_{i\in C}m(C\setminus i)^{a|B\cap C|}}{\prod_{i\in B\cap C} m(C\setminus i )^a}\cdot \frac{1}{|L\cap\pi^{-1}(p)|}\\
        &=\frac{m(X)^{a(|A|+|B|-r+1)}\prod_{i\in C}m(C\setminus i)^{a(|B\cap C|+|A\cap C|-n+1)}}{m(C)^am(A)^a\prod_{i\in A\cap C}m(C\setminus i)^a\prod_{i\in B\cap C}m(C\setminus i)^a}\\
        &=\frac{m(X)^a\prod_{i\in C}m(C\setminus i)^a}{m(C)^am(A)^a\prod_{i\in(A\cup B)\cap C}m(C\setminus i)^a}\\
        &=\frac{m(X)^am((A\cup B)\cap C)^a}{m(C)^am(A)^a}=\frac{m(A\cup B)^a}{m(A)^a}
    \end{align*}
    In this last computation the final equality comes from property \ref{item:AM3} of arithmetic matroids.
\end{proof}

\begin{lemma}
    \label{lemma:prod_omega_in_cohomology}
    We have
        \begin{equation}
    \label{eq:prod_gen_cohomology}
        \omega_{W,A}\omega_{W',A'}=\begin{cases}
            0 \quad \text{ if } A,A' \text{ are not disjoint or } A\sqcup A' \text{ dependent,}\\
            (-1)^{d\ell(A, A')}\sum_{L \text{ c.c.\ of } W\cap W'}\omega_{L,A\sqcup A'} \quad \text{ otherwise.}
        \end{cases}
    \end{equation}
\end{lemma}
\begin{proof}
    We start with the case of $A\sqcup A'$ of nullity at most one.
    Consider any separating cover $\pi \colon U \to G^r$ of $A \sqcup A'$, we have 
    \begin{align*}
        \pi^*(\omega_{W,A}\omega_{W',A'}) &= \sum_{L \in \pi_0(\pi^{-1}(W))} \omega_{A}^U(q_L) \sum_{L' \in \pi_0(\pi^{-1}(W'))} \omega_{A'}^U(q_{L'}) \\
        &= \sum_{K \in \pi_0(\pi^{-1}(W \cap W'))} \omega_A^U(q_K) \omega_{A'}^U(q_K) \\
        &= (-1)^{d\ell(A,A')} \sum_{L \in \pi_0(W \cap W')} \sum_{K \in \pi_0(\pi^{-1}(L))} \omega_{A \sqcup A' }^U(q_K) \\
        &= (-1)^{d\ell(A,A')} \sum_{L \in \pi_0(W \cap W')} \pi^*( \omega_{L,A \sqcup A' })
    \end{align*}
    if $A \cap A'= \emptyset$, otherwise in the second line above we have $\omega_A^U(q_K) \omega_{A'}^U(q_K)=0$ because $(\omega_i^U(q_K))^2=0$.
    By injectivity of $\pi_*$ and Corollary \ref{cor:omega_C_nullo} the claimed equality follows.

    Now consider arbitrary independent sets $A$ and $A'$ and fix a subset $A \subset X \subseteq A \cap A'$ of nullity one.
    Consider a separating cover $f \colon U \to G^r$ for $A'$ and a separating cover $g \colon V \to U$ for $f^{-1}(X)$.
    Notice that $f \circ g$ is a separating cover for $X$.
    We have 
    \[ f^*(\omega_{W',A'})= \sum_{L' \in \pi_0(f^{-1}(W'))} \omega_{A'}^U(q_{L'})\]
    and it is enough to prove that $f^*(\omega_{W,A})\omega_{A' \cap X}^U(q_{L'})=0$ for all $L' \in \pi_0(f^{-1}(W'))$.
    We apply $g^*$ to obtain
    \[ g^*(f^*(\omega_{W,A})\omega_{A' \cap X}^U(q_{L'}))= \sum_{L \in \pi_0(g^{-1}f^{-1}(W))} \omega_A^V(q_L) \sum_{K \in \pi_0 (g^{-1}(L'))} \omega_{A'\cap X}^V(q_K)\]
    but all products $\omega_A^V(q_L)\omega_{A'\cap X}^V(q_K)$ are zero because $X$ is a dependent set, see Corollary \ref{cor:omega_C_nullo}.
    This complete the proof.
\end{proof}

\begin{lemma}\label{lemma: circuitrel} Let $X=C\sqcup F$ be a central set of nullity $1$ and let $C$ be the unique oriented circuit contained in $X$, $Y$ be a connected component of $\bigcap_{i\in X}H_i$, then
\begin{itemize}
\item If $b=1$,
        \begin{align}\label{eq:nonuni1}
        \begin{split}
        \sum_{\substack{K\subseteq C^- \\
        K\neq\emptyset}}(-1)^{|K|}c_{i_K}^d\frac{m(X\setminus K)^a}{m(X\setminus i_K)^a}\eta_{W,X\setminus K,K\setminus i_K}&\\
        -\sum_{\substack{K\subseteq C^+\\ K\neq \emptyset}}(-1)^{|K|}c_{i_K}^d\frac{m(X\setminus K)^a}{ m(X\setminus i_K)^a}\eta_{W,X\setminus K, K\setminus i_K}&=0,
        \end{split}
        \end{align}
        where $i_K\in K$ and, for each summand, $W$ is the connected component of $\bigcap_{i\in X\setminus K}H_i$ such that $Y\subseteq W$.
\item If $b>1$ and $d$ odd,
    \begin{align}\label{eq: nonuni2}
        \sum_{i\in C}(-1)^{|C_{<i}|}
        \omega_{Y,X\setminus i}=0.
    \end{align}
\item If $b>1$ and $d$ even,
    \begin{align}\label{eq: nonuni3}
        \sum_{i\in C^-}\omega_{Y,X\setminus i} - \sum_{i\in C^+}\omega_{Y,X\setminus i}=0.
    \end{align}
\end{itemize}
\end{lemma}

\begin{proof}
We apply Lemma \ref{lemma2} to the circuit $C$ in the abelian variety $U=\Hom (\Lambda',G)$. Let $p\in Y$ be any point.
If $b=1$, we have:
    \begin{align*}
        \sum_{\substack{K\subseteq C^-\\ K\neq \emptyset}}(-1)^{|K|}c_{i_K}^d\eta^U_{C\setminus K,K\setminus i_K}(q)-\sum_{\substack{K\subseteq C^+\\ K\neq \emptyset}}(-1)^{|K|}c_{i_K}^d\eta^U_{C\setminus K, K\setminus i_K}(q)=0
    \end{align*}
    for all $q\in \pi^{-1}(p)$. We multiply this equation by $\omega^U_F(q)$, and we get:
        \begin{align*}
        \sum_{\substack{K\subseteq C^-\\ K\neq \emptyset}}(-1)^{|K|}c_{i_K}^d\eta^U_{X\setminus K,K\setminus i_K}(q)-\sum_{\substack{K\subseteq C^+\\ K\neq \emptyset}}(-1)^{|K|}c_{i_K}^d\eta^U_{X\setminus K,K\setminus i_k}(q)=0.
    \end{align*}
    Summing over all $q\in\pi^{-1}(p)$, we obtain:
    \begin{align*}
        0=&\sum_{q\in\pi^{-1}(p)}\sum_{\substack{K\subseteq C^-\\ K\neq \emptyset}}(-1)^{|K|}c_{i_K}^d\eta^U_{X\setminus K,K\setminus i_K}(q) 
        -\sum_{q\in\pi^{-1}(p)}\sum_{\substack{K\subseteq C^+\\ K\neq\emptyset}}(-1)^{|K|}c_{i_K}^d\eta^U_{X\setminus K,K\setminus i_K}(q)\\
        =&\sum_{\substack{K\subseteq C^-\\ K\neq\emptyset}}(-1)^{|K|}c_{i_K}^d\sum_{q\in\pi^{-1}(p)}\eta^U_{X\setminus K,K\setminus i_K}(q)-\sum_{\substack{K\subseteq C^+\\ K\neq\emptyset}}(-1)^{|K|}c_{i_K}^d\sum_{q\in\pi^{-1}(p)}\eta^U_{X\setminus K, K\setminus i_K}(q)\\
        =&\sum_{\substack{K\subseteq C^-\\ K\neq \emptyset}}(-1)^{|K|}c_{i_K}^d\frac{m(X\setminus K)^a}{m(X\setminus i_K)^a}\pi^*(\eta_{W,X\setminus K,K\setminus i_K})\\
        &-\sum_{\substack{K\subseteq C^+\\ K\neq\emptyset}}(-1)^{|K|}c_{i_K}^d\frac{m(X\setminus K)^a}{ m(X\setminus i_K)^a}\pi^*(\eta_{W,X\setminus K, K\setminus i_K}).
    \end{align*}
    Since $\pi^*$ is injective, equation \eqref{eq:nonuni1} follows. 

    If $b>1$, through an analogous computation, equations \eqref{eq: nonuni2} and \eqref{eq: nonuni3} hold.

\end{proof}

\begin{example}\label{ex: NCNU}
\begin{figure}
    \centering
\begin{tikzpicture}
    \draw[black, thick] (-1,0) -- (3,0);
    \draw[black, thick] (0,-2) -- (0,2);
    \draw[black, thick] (-1,2) -- (1,-2);
    \draw[black, thick] (1,2) -- (3,-2);
    \draw[->] (2.8,0) -- (2.8,0.25); {}
    \draw[->] (0,1.8) -- (0.25,1.8); {}
    \draw[->] (-0.8,1.6) -- (-0.6,1.7); {}
    \draw[->] (1.2,1.6) -- (1.4,1.7); {}
    \node at (-0.3,-1.7){$H_1$};
    \node at (2.8,-0.3){$H_2$};
    \node at (1,1.4){$H_3$};
    
\end{tikzpicture}
\quad \quad \quad \quad
\begin{tikzpicture}[scale=1.3]
\draw[cyan, thick] (-1.5,-1.5) -- (-1.5,1.5);
\draw[black, thick] (1.5,-1.5) -- (1.5,1.5);
\draw[black, thick] (-1.5,-1.5) -- (1.5,-1.5);
\draw[black, thick] (-1.5,1.5) -- (1.5,1.5);
\draw[red, thick] (-1.5,1.5) -- (0,-1.5);
\draw[red, thick] (0,1.5) -- (1.5,-1.5);
\draw[green, thick] (-1.5,-1.5) -- (1.5,-1.5);
\end{tikzpicture}
    \caption{A picture of the non-unimodular circuit $\{1,2,3\}$ in Example \ref{ex: NCNU} when $(a,b)=(1, 0)$ on the left, and $(a,b)=(0,1)$ on the right.}
    \label{fig:Example3}
\end{figure}
    Let us consider the arrangement $\mathcal{A}=\{H_i\}_{i=1,2,3,4}$ in $G^3$ associated to the matrix
    $$\begin{pmatrix}
        1 & 0 & 2 & 1 \\
        0 & 1 & 1 & 0 \\
        0 & 0 & 0 & 2
    \end{pmatrix}$$
    (see Figure \ref{fig:Example3}).
    The sets of nullity $1$ are $C=C^+\sqcup C^-=\{1,2\}\sqcup \{3\}$ and $X=C\sqcup F=\{1,2,3\}\sqcup \{4\}$.

    Let $W$ be the layer $H_1\cap H_2\cap H_3$. If $b>1$, applying formula \eqref{eq:nonuni1} to the central circuit $C$ we obtain the relation
    \begin{equation}\label{NCNU1}
        \omega_{W,12}+(-1)^d\omega_{W,23}-\omega_{W,13}+\frac{1}{2^a}\omega_3\psi_2=0.
    \end{equation}
    Notice that in the case $a=1$ the unimodular covering of $\mathcal{A}$ is the arrangement in Example \ref{ex: NCU}, and the pullback of eq.\ \eqref{NCNU1} is the sum of equations \eqref{CU1} and \eqref{NCU}.

    Now consider a set $X$ of nullity $1$, we get $2^a$ relations because the intersection $H_1\cap H_2\cap H_3 \cap H_4=W\cap H_4$ has $2^a$ distinct connected components, $p_1, \cdots, p_{2^a}$ and for any $i=1,\cdots, 2^a$:
    \begin{equation}\label{NCNU2}
        \omega_{p_i,124}+(-1)^d\omega_{p_i,234}-\omega_{p_i,134}-\frac{1}{4^a}\omega_{Z,34}\psi_2=0,
    \end{equation}
    where $Z$ is the layer $H_3\cap H_4$.
    The Poincaré polynomial of the complement is
    $(t+1)^{3a}+4(t+1)^{2a}+7(t+1)^a+6$.
\end{example}

\begin{definition}
    Let $R$ be the free $H^*(G^r)$-module generated by the classes $\omega_{W,A}\in H^*(M(\mathcal{A}))$ with $A\subseteq E$ independent set and $W$ connected component of $\bigcap_{i\in A}H_i$. We endow $R$ with a ring structure by defining the multiplication
    \begin{equation}
    \label{eq:def_prod_R}
        \omega_{W,A}\omega_{W',A'}=\begin{cases}
            0 \quad \text{ if } A,A' \text{ are not disjoint or } A\sqcup A' \text{ dependent,}\\
            (-1)^{d\ell(A, A')}\sum_{L \text{ c.c. of } W\cap W'}\omega_{L,A\sqcup A'} \quad \text{ otherwise.}
        \end{cases}
    \end{equation}
\end{definition}

    Notice that, for any $A,B\subseteq E$ and for any connected component $W$ of $\bigcap_{i\in A}H_i$, the element $\eta_{W,A,B}\in R$ is defined as $\omega_{W,A}\psi_B$.

\begin{theorem} \label{thm: main}
Let $\mathcal{A}$ be an arrangement in $G^r$, where $G=\mathbb{R}^b\times (S^1)^a$. The integer cohomology of the complement $H^*(M(\mathcal{A});\mathbb{Z})$ is the quotient of $R$ by  the following relations:
\begin{itemize}   
    \item 
    For any independent set $A\subseteq E$ and any connect component $W$ of $\cap_{i \in A} H_i$ and any $z \in \ker \left( H^*(G^r;\mathbb{Z}) \to H^*(W;\mathbb{Z}) \right)$,
    \begin{align}\label{eq: prod1}
       \omega_{W,A}z=0
    \end{align}
  
    \item For any central set $X=C\sqcup F$ of nullity $1$ with associated signed circuit $C=C^+\sqcup C^-$, and for any connected component $Y$ of $\bigcap_{i\in X}H_i$,
    \begin{itemize}
        \item If $b=1$,
        \begin{align}\label{eq: circuit1}
        \begin{split}
        \sum_{\substack{K\subseteq C^- \\ K\neq\emptyset}} & (-1)^{|K|} c_{i_K}^d \frac{m(X\setminus K)^a}{m(X\setminus i_K)^a}\eta_{W,X\setminus K,K\setminus i_K}\\
        &-\sum_{\substack{K\subseteq C^+\\ K\neq \emptyset}}(-1)^{|K|}c_{i_K}^d\frac{m(X\setminus K)^a}{ m(X\setminus i_K)^a}\eta_{W,X\setminus K, K\setminus i_K}=0,
        \end{split}
        \end{align}
        for some $i_K\in K$ and where, for each summand, $W$ is the connected component of $\bigcap_{i \in X\setminus K}H_i$ such that $Y\subseteq W$.
        \item If $b>1$ and $d$ odd,
        \begin{align}\label{eq: circuit2}
        \sum_{i\in C}(-1)^{|C_{<i}|}
       \eta_{Y,X\setminus i,\emptyset}=0.
        \end{align}
        \item If $b>1$ and $d$ even,
        \begin{align}\label{eq: circuit3}
        \sum_{i\in C^-}\eta_{Y,X\setminus i,\emptyset}-\sum_{i\in C^+}\eta_{Y,X\setminus i,\emptyset}=0.
        \end{align}
    \end{itemize}
\end{itemize}
\end{theorem}

\begin{remark}
    Its easy to see that eq.\ \eqref{eq: prod1} implies $\omega_{W,A}\psi_B^j=0$ for all $j=1, \dots, a$ and all $B \subseteq E$ dependent on $A$.
    The viceversa, for essential arrangements, holds only under the assumption that $\AAA$ is unimodular or after tensoring with $\mathbb{Q}$.
\end{remark}
\begin{proof}[of \Cref{thm: main}]
    Consider the map $p\colon R \to H^*(M(\AAA);\mathbb{Z})$ given by the pullback of $M(\AAA) \to G^r$ and that sends the formal generators $\omega_{W,A}$ to the cohomology classes $\omega_{W,A}$.
    The map $p$ is a ring homomorphism because the product in $R$ (defined in eq.\ \eqref{eq:def_prod_R}) coincides with the product in cohomology (see Lemma \ref{lemma:prod_omega_in_cohomology}).
    
    Let $I$ be the ideal generated by relations \eqref{eq: prod1} and \eqref{eq: circuit1} (or \eqref{eq: circuit2} or \eqref{eq: circuit3} depending on $b$ and $d$) and $J$ be the ideal generated by $\eqref{eq: prod1}$ only. 
    The ideal of relations $I$ is contained in $\ker p$: indeed, Relation \eqref{eq: prod1} holds by Lemma \ref{lemma: rel1} and Relations \eqref{eq: circuit1}, \eqref{eq: circuit2}, \eqref{eq: circuit3} by Lemma \ref{lemma: circuitrel}.

    We prove that $p$ induces an isomorphism $R/I \to H^*(M(\AAA);\mathbb{Z})$ by induction on the cardinality of $\mathcal{A}$. The base case $\mathcal{A}= \emptyset$ is trivial. 
    
    Let $S= R/J$ and fix $i\in E$. Let $S'$,  $S''$, $I'$, $I''$ 
    be the analog objects with respect to the deletion $\mathcal{A}'$ and the contraction $\mathcal{A}''$ of the subvariety $H_i$.
    Consider the short exact sequence
    \[ 0 \to S' \to S \xrightarrow[]{g} S'' \to 0\]
    given by the inclusion and the residue map $g$ defined by 
    \[
    g(\omega_{W,A}z) = \begin{cases}
    \omega_{W, A\setminus i}\bar{z}, &\text{ if } i\in A,\\
    0 &\text{ otherwise,}
    \end{cases}
    \]
    for any $z\in H^*(G^r)$, where $\bar{z}$ is the image of $z$ through the restriction map $H^*(G^r)\rightarrow H^*(H_i)$.
    Checking that the maps are well defined and the sequence is exact is a routine exercise, that can be carried out using the key fact that $J$ is generated by relations $\eqref{eq: prod1}$ as a $\mathbb{Z}$-module.

    With an abuse of notation, we call $I$ the image of the ideal $I$ in the quotient ring $S$.
    Consider the diagram
    \begin{center}
    \begin{tikzcd}
             & 0 \arrow[d] & 0 \arrow[d] & 0 \arrow[d]& \\
        & I' \arrow[r] \arrow[d] & \ker p \arrow[r] \arrow[d]       & I'' \arrow[d]      &   \\
        0 \arrow[r] & S' \arrow[r] \arrow[d, "p'"] & S \arrow[r, "g"] \arrow[d, "p"]       & S'' \arrow[d, "p''"]      \arrow[r] & 0   \\
        0 \arrow[r] & H^*(M(\mathcal{A'})) \arrow[r, "j^{*}"]  \arrow[d]    & H^*(M(\mathcal{A})) \arrow[r, "\res"]  & H^{*-(a+b-1)}(M(\mathcal{A''})) \arrow[r] \arrow[d] & 0  \\
         & 0  &  & 0 & 
    \end{tikzcd}  
    \end{center}
    
    By induction hypothesis the first and last columns are exact. The second row is exact by the previous considerations and the last one by Corollary \ref{cor: ses}.
    The left square of diagram commutes because $j^*(\eta'_{W,A,B})=\eta_{W,A,B}$, as the right one, since $\res(\eta_{W,A,B})=\eta_{W,A\setminus\{i\},B}^{''}$ if $i \in A$ and $\res(\eta_{W,A,B})=0$ otherwise.

    The Snake Lemma implies that the map $p$ is surjective, and that 
    \[ 0 \to I' \to \ker p \to I'' \to 0 \]
    is exact.

    Now we prove $\ker p = I$, we already noticed that $\ker p \supseteq I$.
    Regarding the other inclusion, observe that for any relation $r$ of type \eqref{eq: circuit1} and any generator $\eta_{W,A,B}$ the element $r\eta_{W,A,B}$ is linear combination of relations of types \eqref{eq: prod1} and \eqref{eq: circuit1}.
    In particular, the ideal $I$ is generated by \eqref{eq: prod1} and \eqref{eq: circuit1} as $\mathbb{Z}$-module.
    It follows that the map $g_{|I} \colon I\twoheadrightarrow I''$ is surjective.
    Moreover, $I' \subseteq I$ and this implies $\ker p \subseteq I$.

    This completes the inductive step and concludes the proof.
\end{proof}

\section{Relation with \texorpdfstring{\protect\cite{CDDMP2020}}{CDD+20}}
\label{sec:6}
    Our presentation, specialized to toric arrangements, is similar to the one obtained in \cite{CDDMP2020}.
    The integral cohomology ring was studied in \cite[Section 7]{CDDMP2020} without providing an explicit presentation, but showing that the cohomology is generated by the differential forms $\frac{m(A)}{m(A\cup B)}\eta_{W,A,B}$.
    Those cohomological classes $\eta_{W,A,B}$ correspond to our classes $\eta_{W,A,B}$; recall that they depend on the choice of an orientation of each hypertorus.

    The rational cohomology was presented in \cite[Theorem 6.13]{CDDMP2020} using the differential forms $\overline{\eta}_{W,A,B}$.
    That differential forms are an average of $\eta_{W,A,B}$ over all possible orientation choice
    \begin{equation}
    \label{eq:rel_etabar_eta}
    \overline{\eta}_{W,A,B} = \sum_{D \subseteq A} (-1)^{\lvert D \rvert} 2^{\lvert A \setminus D \rvert} \frac{m(A \setminus D)}{m(A)} \eta_{W(A\setminus D), A \setminus D, B \cup D}.    
    \end{equation}
    Therefore, their eq.\ (28) in \cite[Theorem 6.13]{CDDMP2020} corresponds to our multiplication rule in $R$ eq.\ \eqref{eq:def_prod_R} and to eq.\ \eqref{eq: prod1}, and their eq.\ (29) holds in $H^*(G^r) \subset R$.
    The eq.\ (30) is equivalent to eq.\ \eqref{eq: circuit1} in a very intricate way that we are going to explain.

    Consider $X$ a subset of nullity $1$, $C=C^+ \sqcup C^-$ be the  oriented circuit in $X$ and $Y$ a connected component of $\cap_{i \in X} H_i$.
    We substitute eq.\ \eqref{eq:rel_etabar_eta} in their eq.\ (30):
    \begin{align*}
         (30) &= \sum_{i \in C} \sum_{\substack{B \subseteq C \setminus \{i\} \\ \lvert B \rvert \text{ even}}} (-1)^{\lvert A_{<i} \rvert + \lvert B \cap C^- \rvert} \frac{m(A)}{m(X\setminus \{i\})} \overline{\eta}_{W(A),A,B}\\
        &= \sum_{i \in C} \sum_{\substack{B \subseteq C \setminus \{i\} \\ \lvert B \rvert \text{ even}}} \sum_{D \subseteq A} (-1)^{\lvert A_{<i} \rvert + \lvert B \cap C^- \rvert+\lvert D \rvert} 2^{\lvert A \setminus D \rvert} \frac{m(A\setminus D)}{m(X\setminus \{i\})} \eta_{W(A\setminus D), A \setminus D, B \cup D},
    \end{align*}
    where $A= X \setminus (B \cup \{i\})$ and $W(A)$ is the unique connected component of $\cap_{i \in A} H_i$ containing $Y$.

    Let $E=A \setminus D$ and notice that for any $i,j \in C$ we have
    \begin{equation}
    \label{eq:rel_etai_etaj}
        c_i \frac{\eta_{W(E),E,X \setminus (E \cup \{i\})}}{m(X \setminus \{i\})} = c_j \frac{\eta_{W(E),E,X \setminus (E \cup \{j\})}}{m(X \setminus \{j\})}.
    \end{equation}
    Substituting eq.\ \eqref{eq:rel_etai_etaj} we have
    \begin{align*}
        &(30)= \sum_{\substack{E \subset X}} \sum_{\substack{B \subseteq C \setminus E \\ \lvert B \rvert \text{ even}}} \sum_{i \in C \setminus (E \cup B)} (-1)^{\lvert A_{<i} \rvert + \lvert B \cap C^- \rvert+\lvert D \rvert} 2^{\lvert E \rvert} \frac{m(E)}{m(X\setminus \{i\})} \eta_{W(E), E, X \setminus (E \cup \{i\})} \\
        &= \sum_{E \subset X}  (-1)^{\lvert X \setminus E \rvert-1} 2^{\lvert E \rvert} \frac{m(E) c_j \eta_{W(E), E, X \setminus (E \cup \{j\})}}{m(X\setminus \{j\})} 
        \sum_{\substack{B \subseteq C \setminus E \\ \lvert B \rvert \text{ even}}} \sum_{i \in C \setminus (E \cup B)} (-1)^{\lvert (E \sqcup D)_{<i} \rvert + \lvert B \cap C^- \rvert}c_i,
    \end{align*}
    where $j$ is any element in $X \setminus E$.
    Notice that $c_i= (-1)^{\lvert X_{<i} \rvert + \lvert \{i\} \cap C^- \rvert} $ and we claim that
    \begin{equation}
        \label{eq:sum_sign_equal_power_of_two}
        \sum_{\substack{B \subseteq C \setminus E \\ \lvert B \rvert \text{ even}}} \sum_{i \in C \setminus (E \cup B)} (-1)^{\lvert B_{<i} \rvert -1 + \lvert (B\cup \{i\}) \cap C^- \rvert } = \begin{cases}
            -2^{\lvert C \setminus E \rvert -1} & \text{if } \emptyset \neq C\setminus E \subseteq C^+, \\
            2^{\lvert C \setminus E \rvert -1} & \text{if } \emptyset \neq C\setminus E \subseteq C^-, \\
            0 & \text{otherwise.}
        \end{cases}
    \end{equation}
    The left hand side of eq.\ \eqref{eq:sum_sign_equal_power_of_two} is equal to
    \[ \sum_{\substack{\tilde{B} \subseteq C \setminus E \\ \lvert \tilde{B} \rvert \text{ odd}}} (-1)^{\lvert \tilde{B} \cap C^- \rvert-1} \]
    where $\tilde{B}= B \sqcup \{i\}$, and so proving eq.\ \eqref{eq:sum_sign_equal_power_of_two} becomes an exercise left to the reader.

    Finally, we have
    \begin{align*}
        \begin{split}
            (30) = &\sum_{\substack{E \subset X\\ \emptyset \neq  C \setminus E \subseteq C^-}}  (-1)^{\lvert X \setminus E \rvert-1} 2^{\lvert E \cup C\rvert-1} \frac{m(E) c_j \eta_{W(E), E, X \setminus (E \cup \{j\})}}{m(X\setminus \{j\})} \quad \\ 
            &- \sum_{\substack{E \subset X\\ \emptyset \neq  C \setminus E \subseteq C^+}}  (-1)^{\lvert X \setminus E \rvert-1} 2^{\lvert E \cup C\rvert-1} \frac{m(E) c_j \eta_{W(E), E, X \setminus (E \cup \{j\})}}{m(X\setminus \{j\})}
        \end{split} \\
        =& \sum_{F\subseteq X \setminus C} (-1)^{\lvert F \rvert} 2^{\lvert X\setminus F\rvert + \ell (F,X\setminus F)}\frac{m(X\setminus F)}{m(X)} \eta_{T,\emptyset,F} r(W(X\setminus F), X\setminus F),
    \end{align*}
    where $X=F \sqcup (C\cup E)$ and we denote by $r(W(Z), Z)$ the relation \eqref{eq: circuit1} for the subset $Z$ of nullity $1$ and the connected component $W(Z)$.

    This provides an alternative proof of \cite[Theorem 6.13]{CDDMP2020}.

\section{Configuration spaces}\label{Sec: conf}

In this section, we apply the main result of this article to the cohomology ring of ordered configuration spaces  in $\mathbb{R}^b \times (S^1)^a$.

Let $X$ be a topological space, the configuration space of $n$ points in $X$ is:
$$\conf_n(X)=\{(x_1, \dots, x_n)\in X^n \mid x_i\neq x_j \, \forall i\neq j\}.$$
As above, let $G=\mathbb{R}^b\times (S^1)^a$, and consider the totally unimodular and central arrangement 
$$\mathcal{A}_n = \{ H_{ij} \}_{\substack{i,j \in [n] \\ i<j}}, \quad H_{ij} = (\chi_i - \chi_j)^{(-1)} (e),$$
where $\chi_i \colon G^n \rightarrow G$ is the projection on the $i$-th component. Notice that
\[\conf_n(G) = M^{a,b}(\mathcal{A}_n).\]

Cohen and Taylor \cite{CohenTaylor1978} used a spectral sequence to compute the cohomology of configuration spaces in $\mathbb{R}^b \times M$, for $b\geq 1$ and $M$ a connected manifold (see \cite[Example 1]{CohenTaylor1978}).
More precisely, there exists a filtration $F_*$ on $H^*(\conf_n(\mathbb{R}^b \times M); \mathbb{K})$ for a field $\mathbb{K}$ and they described explicitly the associated graded 
\[ \gr_{F_*}H^*(\conf_n(\mathbb{R}^b \times M); \mathbb{K})\]
as a ring. In the case $M=(S^1)^a$ and $b>1$, \Cref{thm: main} implies that 
\[ \gr_{F_*}H^*(\conf_n(G); \mathbb{K}) \simeq H^*(\conf_n(G); \mathbb{K})\]
as ring.
In the case $b=1$, we have the opposite behaviour: indeed, the two rings $\gr_{F_*}H^*(\conf_n(G; \mathbb{K}))$ and $H^*(\conf_n(G); \mathbb{K})$ are not canonically isomorphic.
Finally, our result extend the coefficients from fields to integers.

Now, we write down the presentation of $H^*(\conf_n(G); \mathbb{Z})$ by using \Cref{thm: main}.
The ground set of the associated matroid is the set $E=\{ij | i,j\in [n], i< j\}$ with the lexicographic order. The poset of layers is $\mathcal{L}(\mathcal{A}_n)=\Pi_{[n]}$ the poset of partitions of $[ n]$. The circuits of the arrangement are of the form $\{ij, ik, jk\}$ with $i,j,k\in [n]$ and $i<j<k$.

By Theorem \ref{thm: main}, the cohomology of the configuration space is generated as a $H^*(G^n)$-module by the classes $\omega_{ij} \in H^d(\conf_n(G))$ with $ij \in E$.\\
Relations \eqref{eq: prod1}-\eqref{eq: circuit2} become
\begin{itemize}
    \item $\omega_{ij}\psi_{ij}=0$ for all $ ij \in E$;
    \item for $b=1$, \begin{equation}\label{eq: arnold_corretta}
        \omega_{ij}\omega_{jk}-\omega_{ij}\omega_{ik}+\omega_{jk}\omega_{ik} -\psi_{ij}\omega_{ik}=0;
    \end{equation}
    \item for $b>1$, 
    \begin{equation}
        \label{eq: arnold}
        \omega_{ij}\omega_{jk}-\omega_{ij}\omega_{ik}+\omega_{jk}\omega_{ik}=0.
    \end{equation}
\end{itemize}

Recall that in $H^*(G^n)$, $\psi_{ik}=\psi_{ij}+\psi_{jk}$ and that $\omega_{ij}^2=0$ by the product rule. 
Moreover the relation $\omega_{ij}\omega_{jk}\omega_{ik}=0$ holds multiplying equation \eqref{eq: arnold_corretta} (resp.\ eq.\ \eqref{eq: arnold})  by $\omega_{ik}$.  
In the case $b=1$, correction terms appear in our formulas. In the article \cite{CohenTaylor1978} and in this one, the element $\omega_{ij}$ represent the class of the subvariety $H_{ij}$ and hence there is no canonical morphism between the two algebras $\gr_{F_*}H^*(\conf_n(G; \mathbb{K}))$ and $H^*(\conf_n(G); \mathbb{K})$ for $n \geq 3$.

\section*{Acknowledgements}
R.\ Pagaria and M.\ Pismataro are partially supported by PRIN ALgebraic and TOPological combinatorics (ALTOP) n° 2022A7L229.
The authors are member of the INDAM group GNSAGA, that has partially supported R.\ Pagaria and M.\ Pismataro.

\bibliographystyle{amsalpha}
\bibliography{bib}

\newcommand{\etalchar}[1]{$^{#1}$}
\providecommand{\bysame}{\leavevmode\hbox to3em{\hrulefill}\thinspace}
\providecommand{\MR}{\relax\ifhmode\unskip\space\fi MR }
\providecommand{\MRhref}[2]{%
  \href{http://www.ams.org/mathscinet-getitem?mr=#1}{#2}
}
\providecommand{\href}[2]{#2}
\begin{thebibliography}{BLVS{\etalchar{+}}99}

\bibitem[And25]{AndersonOrientedMatroids}
Laura Anderson, \emph{Oriented matroids}, Cambridge Studies in Advanced
  Mathematics, vol. 216, Cambridge University Press, Cambridge, 2025.
  \MR{4880415}

\bibitem[Arn69]{Arnold69}
Vladimir~I. Arnold, \emph{The cohomology ring of the group of dyed braids},
  Mat. Zametki \textbf{5} (1969), 227--231. \MR{242196}

\bibitem[Bib16]{Bibby}
Christin Bibby, \emph{Cohomology of abelian arrangements}, Proc. Amer. Math.
  Soc. \textbf{144} (2016), no.~7, 3093--3104. \MR{3487239}

\bibitem[BLVS{\etalchar{+}}99]{OrientedMatroids}
Anders Bj\"{o}rner, Michel Las~Vergnas, Bernd Sturmfels, Neil White, and
  G\"{u}nter~M. Ziegler, \emph{Oriented matroids}, second ed., Encyclopedia of
  Mathematics and its Applications, vol.~46, Cambridge University Press,
  Cambridge, 1999. \MR{1744046}

\bibitem[BM14]{BradenMoci}
Petter Br\"and\'en and Luca Moci, \emph{The multivariate arithmetic {T}utte
  polynomial}, Trans. Amer. Math. Soc. \textbf{366} (2014), no.~10, 5523--5540.
  \MR{3240933}

\bibitem[Bri72]{Brieskorn1971-1972}
Egbert Brieskorn, \emph{Sur les groupes de tresses}, Séminaire Bourbaki
  \textbf{14} (1971-1972), 21--44 (fre).

\bibitem[CD17]{CD17}
Filippo Callegaro and Emanuele Delucchi, \emph{The integer cohomology algebra
  of toric arrangements}, Adv. Math. \textbf{313} (2017), 746--802.
  \MR{3649237}

\bibitem[CDD{\etalchar{+}}20]{CDDMP2020}
Filippo Callegaro, Michele D'Adderio, Emanuele Delucchi, Luca Migliorini, and
  Roberto Pagaria, \emph{Orlik-{S}olomon type presentations for the cohomology
  algebra of toric arrangements}, Trans. Amer. Math. Soc. \textbf{373} (2020),
  no.~3, 1909--1940. \MR{4068285}

\bibitem[CLM76]{CLM76}
Frederick~R. Cohen, Thomas~J. Lada, and J.~Peter May, \emph{The homology of
  iterated loop spaces}, Lecture Notes in Mathematics, vol. Vol. 533,
  Springer-Verlag, Berlin-New York, 1976. \MR{436146}

\bibitem[Cor02]{Cordovil02}
Raul Cordovil, \emph{A commutative algebra for oriented matroids}, Geometric
  combinatorics (San Francisco, CA/Davis, CA, 2000), vol.~27, 2002, pp.~73--84.
  \MR{1871690}

\bibitem[CT78]{CohenTaylor1978}
Frederick~R. Cohen and Laurence Taylor, \emph{Computations of {G}elfand-{F}uks
  cohomology, the cohomology of function spaces, and the cohomology of
  configuration spaces}, Geometric applications of homotopy theory ({P}roc.
  {C}onf., {E}vanston, {I}ll., 1977), {I}, Lecture Notes in Math., vol. 657,
  Springer, Berlin, 1978, pp.~106--143. \MR{513543}

\bibitem[DB23]{DorpalenBarry23}
Galen Dorpalen-Barry, \emph{The {V}archenko-{G}el'fand ring of a cone}, J.
  Algebra \textbf{617} (2023), 500--521. \MR{4523451}

\bibitem[DBPW22]{DBPW22}
Galen Dorpalen-Barry, Nicholas Proudfoot, and Jidong Wang, \emph{Equivariant
  cohomology and conditional oriented matroids}, 2022.

\bibitem[DCG18]{DCG18}
Corrado De~Concini and Giovanni Gaiffi, \emph{Projective wonderful models for
  toric arrangements}, Adv. Math. \textbf{327} (2018), 390--409. \MR{3761997}

\bibitem[DCG19]{DCG19}
\bysame, \emph{Cohomology rings of compactifications of toric arrangements},
  Algebr. Geom. Topol. \textbf{19} (2019), no.~1, 503--532. \MR{3910589}

\bibitem[DCG21]{DCG21}
\bysame, \emph{A differential algebra and the homotopy type of the complement
  of a toric arrangement}, Atti Accad. Naz. Lincei Rend. Lincei Mat. Appl.
  \textbf{32} (2021), no.~1, 1--21. \MR{4251238}

\bibitem[DCGP20]{DCGP20}
Corrado De~Concini, Giovanni Gaiffi, and Oscar Papini, \emph{On projective
  wonderful models for toric arrangements and their cohomology}, Eur. J. Math.
  \textbf{6} (2020), no.~3, 790--816. \MR{4151719}

\bibitem[DCP95]{DCP95}
Corrado De~Concini and Claudio Procesi, \emph{Wonderful models of subspace
  arrangements}, Selecta Math. (N.S.) \textbf{1} (1995), no.~3, 459--494.
  \MR{1366622}

\bibitem[DCP05]{DeConciniProcesi2005}
\bysame, \emph{On the geometry of toric arrangements}, Transform. Groups
  \textbf{10} (2005), no.~3-4, 387--422. \MR{2183118}

\bibitem[Del72]{Deligne1972}
Pierre Deligne, \emph{Les immeubles des groupes de tresses généralisés.},
  Inventiones mathematicae \textbf{17} (1972), 273--302.

\bibitem[dLS01]{LonguevilleSchultz}
Mark de~Longueville and Carsten~A. Schultz, \emph{The cohomology rings of
  complements of subspace arrangements}, Math. Ann. \textbf{319} (2001), no.~4,
  625--646. \MR{1825401}

\bibitem[DM13]{DAdderioMoci}
Michele D'Adderio and Luca Moci, \emph{Arithmetic matroids, the {T}utte
  polynomial and toric arrangements}, Adv. Math. \textbf{232} (2013), 335--367.
  \MR{2989987}

\bibitem[FZ00]{FeichtnerZiegler}
Eva~Maria Feichtner and G\"{u}nter~M. Ziegler, \emph{On cohomology algebras of
  complex subspace arrangements}, Trans. Amer. Math. Soc. \textbf{352} (2000),
  no.~8, 3523--3555. \MR{1694288}

\bibitem[GM12]{goresky}
Mark Goresky and Robert MacPherson, \emph{Stratified morse theory}, Ergebnisse
  der Mathematik und ihrer Grenzgebiete. 3. Folge / A Series of Modern Surveys
  in Mathematics, Springer Berlin Heidelberg, 2012.

\bibitem[GPS24]{GPS2024}
Lorenzo Giordani, Roberto Pagaria, and Viola Siconolfi, \emph{Cohomology rings
  of toric wonderful model}, 2024.

\bibitem[GR89]{GR89}
Izrail'~Moiseevič Gel'fand and Grigory~L. Rybnikov, \emph{Algebraic and
  topological invariants of oriented matroids}, Dokl. Akad. Nauk SSSR
  \textbf{307} (1989), no.~4, 791--795. \MR{1020668}

\bibitem[HR96]{HughesRanickiBook}
Bruce Hughes and Andrew Ranicki, \emph{Ends of complexes}, Cambridge Tracts in
  Mathematics, vol. 123, Cambridge University Press, Cambridge, 1996.
  \MR{1410261}

\bibitem[Loo93]{Looijenga}
Eduard Looijenga, \emph{Cohomology of {${\mathscr M}_3$} and {${\mathscr
  M}^1_3$}}, Mapping class groups and moduli spaces of {R}iemann surfaces
  ({G}\"ottingen, 1991/{S}eattle, {WA}, 1991), Contemp. Math., vol. 150, Amer.
  Math. Soc., Providence, RI, 1993, pp.~205--228. \MR{1234266}

\bibitem[LTY21]{LiuTranYoshinaga2021}
Ye~Liu, Tan~Nhat Tran, and Masahiko Yoshinaga, \emph{{$G$}-{T}utte polynomials
  and abelian {L}ie group arrangements}, Int. Math. Res. Not. IMRN (2021),
  no.~1, 152--190. \MR{4198494}

\bibitem[Mos17]{Moseley17}
Daniel Moseley, \emph{Equivariant cohomology and the {V}archenko-{G}elfand
  filtration}, J. Algebra \textbf{472} (2017), 95--114. \MR{3584871}

\bibitem[MP22]{MociPagaria}
Luca Moci and Roberto Pagaria, \emph{On the cohomology of arrangements of
  subtori}, J. Lond. Math. Soc. (2) \textbf{106} (2022), no.~3, 1999--2029.
  \MR{4498547}

\bibitem[OS80]{OrlikSolomon1980}
Peter Orlik and Louis Solomon, \emph{Combinatorics and topology of complements
  of hyperplanes}, Invent. Math. \textbf{56} (1980), no.~2, 167--189.
  \MR{558866}

\bibitem[Oxl11]{OxleyMatroidTheory}
James Oxley, \emph{Matroid theory}, second ed., Oxford Graduate Texts in
  Mathematics, vol.~21, Oxford University Press, Oxford, 2011. \MR{2849819}

\bibitem[Pag20]{OrientableArithmeticMatroids}
Roberto Pagaria, \emph{Orientable arithmetic matroids}, Discrete Math.
  \textbf{343} (2020), no.~6, 111872, 8. \MR{4067961}

\bibitem[PP21]{PagariaPaolini}
Roberto Pagaria and Giovanni Paolini, \emph{Representations of torsion-free
  arithmetic matroids}, European J. Combin. \textbf{93} (2021), Paper No.
  103272, 17. \MR{4186616}

\bibitem[Pro06]{Proudfoot06}
Nicholas Proudfoot, \emph{The equivariant {O}rlik-{S}olomon algebra}, J.
  Algebra \textbf{305} (2006), no.~2, 1186--1196. \MR{2266876}

\bibitem[VG87]{VarchenkoGelfand87}
Alexander~N. Varchenko and Izrail'~Moiseevič Gel'fand, \emph{Heaviside
  functions of a configuration of hyperplanes}, Funktsional. Anal. i Prilozhen.
  \textbf{21} (1987), no.~4, 1--18, 96. \MR{925069}

\end{thebibliography}
\end{document}